\documentclass[reqno,11pt]{amsart}

\usepackage{amssymb}
\usepackage{amsmath}
\usepackage{amsfonts}
\usepackage{bm}
\usepackage{caption,subcaption}
\usepackage{graphicx}
\usepackage{amsthm}
\usepackage{enumerate}
\usepackage[mathscr]{eucal}
\usepackage{float}
\usepackage{mathrsfs}
\usepackage{multicol}
\usepackage{multirow}
\usepackage[all,pdf]{xy}
\usepackage[a4paper,left=3cm,right=3cm]{geometry}
\usepackage[table,xcdraw]{xcolor} 
\usepackage{tikz-cd}
\usepackage{quiver} 
\usepackage{resizegather}
\usepackage{diagbox}
\usepackage[bookmarks=false]{hyperref}
\usepackage{varwidth}
\hypersetup{
colorlinks=true,
linkcolor=blue,
urlcolor=blue,
}

\makeatletter
\@namedef{subjclassname@2010}{%
  \textup{2010} Mathematics Subject Classification}
\makeatother

\numberwithin{equation}{section}

\theoremstyle{plain}
\newtheorem{theorem}{Theorem}[section]
\newtheorem{lemma}[theorem]{Lemma}
\newtheorem{proposition}[theorem]{Proposition}
\newtheorem{corollary}[theorem]{Corollary}
\newtheorem{conj}[theorem]{Conjecture}
\newtheorem{defn}[theorem]{Definition}

\newtheorem{eg}[theorem]{Example}
\newtheorem*{notation}{Conventions and Notations}

\theoremstyle{plain}

\numberwithin{equation}{section}

\theoremstyle{remark}

\newtheorem{remark}[theorem]{Remark}

\DeclareMathOperator{\Img}{\operatorname{Im}}
\DeclareMathOperator{\Id}{\operatorname{Id}}
\DeclareMathOperator{\rep}{\operatorname{rep}}
\DeclareMathOperator{\Mod}{\operatorname{mod}}
\DeclareMathOperator{\Hom}{\operatorname{Hom}}
\DeclareMathOperator{\Ext}{\operatorname{Ext}}

\DeclareMathOperator{\Spec}{\operatorname{Spec}}

\DeclareMathOperator{\Homup}{\overline{\operatorname{Hom}}}

\DeclareMathOperator{\gldim}{\operatorname{gl.dim}}
\DeclareMathOperator{\projdim}{\operatorname{proj.dim}}
\DeclareMathOperator{\injdim}{\operatorname{inj.dim}}
\DeclareMathOperator{\dimv}{\operatorname{\underline{\mathbf{dim}}}}

\DeclareMathOperator{\Flagd}{\operatorname{Flag}_{d}}
\DeclareMathOperator{\Flagdstr}{\operatorname{Flag}_{d,\operatorname{str}}}
\newcommand{\Gr}{\operatorname{Gr}}
\newcommand{\Grr}{\operatorname{Gr}}
\newcommand{\Gralg}[1]{\operatorname{Gr}^{#1}}
\newcommand{\Grq}{\operatorname{Gr}^{KQ}}
\newcommand{\Flag}[1]{\operatorname{Flag}_{#1}}
\newcommand{\Flagstr}[1]{\operatorname{Flag}_{#1,\operatorname{str}}}
\newcommand{\dimvec}[1]{\boldsymbol{#1}}
\newcommand{\ord}{\operatorname{ord}}
\newcommand{\orde}{s_{P(e)} }
\newcommand{\shorttimes}{\!\times\!}
\newcommand{\oder}{\;\;\text{ or }\;\;}
\newcommand{\representation}[2]{\genfrac{}{}{0pt}{3}{\phantom{000}#2\phantom{00}}{#1}}
\makeatletter
\newcommand\xleftrightarrow[2][]{%
  \ext@arrow 9999{\longleftrightarrowfill@}{#1}{#2}}
\newcommand\longleftrightarrowfill@{%
  \arrowfill@\leftarrow\relbar\rightarrow}
\makeatother
\usepackage[colorinlistoftodos,textsize=footnotesize]{todonotes}
\begin{document}
\date{}

\title
{Affine Pavings of Quiver Flag Varieties}

\author{Xiaoxiang Zhou}
\address{Institut für Mathematik\\
Humboldt-Universität zu Berlin\\
Berlin, 12489\\ Germany\\} 
\email{email:xiaoxiang.zhou@hu-berlin.de}

\begin{abstract}
In this article, we construct affine pavings for quiver partial flag varieties when the quiver is of Dynkin type. To achieve our results, we extend methods from Cerulli-Irelli--Esposito--Franzen--Reineke
 and Maksimau as well as techniques from Auslander--Reiten theory.
\end{abstract}

\maketitle
\tableofcontents
\section{Introduction}
Affine pavings are an important concept in algebraic geometry similar to cellular decompositions in topology. A complex algebraic variety $X$ has an affine paving if $X$ has a filtration
$$\varnothing= X_0 \subset X_1 \subset \cdots \subset X_d=X$$
with $X_i$ closed and $X_{i+1} \setminus X_i$ isomorphic to some affine space $\mathbb{A}^k_{\mathbb{C}}$.

Affine pavings imply nice properties about the cohomology of varieties, for example the vanishing of cohomology in odd degrees. For other properties see \cite[1.7]{de1988homology}.

Affine pavings have been constructed in many cases, as for Grassmannians \cite[Theorem 4.1]{3264}, flag varieties \cite[Theorem 9.9.5]{Dmoduleperv}, as well as certain Springer fibers \cite{E7paving}, quiver Grassmannians \cite[Theorem 4]{irelli2019cell}, and quiver flag varieties \cite[Theorem 1.2]{maksimau2019flag}. This article focuses on the case of (strict) partial flag varieties which parameterize subrepresentations of a fixed indecomposable representation of a quiver. In particular, we consider quivers of Dynkin type or affine type.
In this case, affine pavings have been constructed in \cite{irelli2019cell} for quiver Grassmannians in all types and in \cite{maksimau2019flag} for partial flag varieties of type $A$ and $D$ (see Table \ref{table:result}). Besides, affine pavings 
have been constructed in \cite[Theorem 6.3]{eberhardt2022motivic} for strict partial flag varieties in type $\tilde{A}$ with cyclic orientation, which generalized the result in \cite{sauter2015cell} for complete quiver flag varieties in nilpotent representations of an oriented cycle. In this paper, we will tackle the remaining cases.

\begin{theorem}
Let $Q$ be a quiver, and let $M$ be a representation of $Q$.
\begin{enumerate}[(1)]
\item If $Q$ is Dynkin, then any (strict) partial flag variety $\Flag{}(M)$ has an affine paving;
\item If $Q$ is of type $\tilde{A}$ or $\tilde{D}$, then for any indecomposable representation $M$, the (strict) partial flag variety $\Flag{}(M)$ has an affine paving;
\item If $Q$ is of type $\tilde{E}$, assume that $\Flag{}(N)$ has an affine paving for any regular quasi-simple representation $N \in \rep(Q)$, then $\Flag{}(M)$ has an affine paving for any indecomposable representation $M$.
\end{enumerate}
\end{theorem}
\renewcommand{\arraystretch}{1.3}

\begin{table}[ht]
\vspace{0.5cm}
\begin{tabular}{|c|c|c|c|}
\hline
            & $\Grq(X)$                  & $\Flagd(X)$                          & $\Flagdstr(X)$          \\ \hline
$A,D$         & \multirow{2}{*}{\cite[Section 5]{irelli2019cell}} & \multirow{1}{*}{\cite[Theorem 1.2]{maksimau2019flag}}        & \multirow{1}{*}{Theorem \ref{thm:Dynkincase}}                                                                   \\ \cline{1-1} \cline{3-4} 
$E$         &                            & \multicolumn{2}{c|}{Theorem \ref{thm:Dynkincase}}                                          \\ \hline
$\tilde{A}, \tilde{D}$ & \multirow{2}{*}{\cite[Section 6]{irelli2019cell}} & \multicolumn{2}{c|}{\multirow{1}{*}{Theorem \ref{thm:affinecase}}}                         \\  \cline{1-1} \cline{3-4} 
$\tilde{E}$ &                            & \multicolumn{2}{c|}{reduced to the regular quasi-finite case.} \\ \hline
\end{tabular}
\vspace{1mm}
\caption{Comparison of existing and new results.}\label{table:result}
\end{table}

We proceed as follows. In Section \ref{sec:flag=gr}, we discuss basic definitions and properties of partial flags. In Section \ref{sec:mainthm} we will prove key Theorems \ref{thm:main1} and \ref{thm:main2}, which allow us to construct affine pavings for quiver partial flag varieties inductively. We apply these theorems to partial flag varieties of Dynkin type, see Section \ref{sec:Dynkin}, and to partial flag varieties of affine type, see Section \ref{sec:affine}. We will combine and extend results from \cite{irelli2019cell} and \cite{maksimau2019flag}.
Following the arguments of \cite{maksimau2019flag} would require studying millions of cases when we  consider the Dynkin quivers of type $E$. To avoid this, we extend the methods of \cite{irelli2019cell} from quiver Grassmannian to quiver partial flag variety. This will reduce the case by case analysis to a feasible computation of (mostly) 8 critical cases, which we carry out in Section \ref{sec:Dynkin} and Section \ref{sec:proofcomplement}. 
%
%
\begin{notation}
Throughout this article, $K=\mathbb{C}$, $R$ is a $K$-algebra with unit, and $\Mod(R)$ denotes the category of left $R$-modules of finite dimension. Let $Q$ be a quiver equipped with the finite set of vertices $v(Q)$ and the finite set of edges $a(Q)$. For an
arrow $b$, we call $s(b)$ the starting vertex and $t(b)$ the terminal vertex
of $b$. We denote by $KQ$ the path algebra and $\rep(Q)=\Mod(KQ)$ the category of quiver representations of finite dimension. For a representation $X\in \rep(Q)$, we denote by $X_i:=e_iX$ the $K$-linear space at the vertex $i\in v(Q)$. We denote by $P(i)$, $I(i)$ and $S(i)$ the indecomposable projective, injective, simple modules corresponding to the vertex $i$, respectively.
\end{notation}

\section*{Acknowledgement}

First, I would like to thank my supervisor, Jens Niklas Eberhardt, for introducing me this specific problem, discussing earlier drafts, and assisting with the write-up. I also thank Hans Franzen for answering some questions regarding \cite{irelli2019cell}, and thank Ruslan Maksimau, Francesco Esposito for their valuable comments and suggestions. I am grateful to the anonymous referee for their constructive feedback, which has greatly improved this work. This article is part of my master thesis project. Submission of this paper was partially supported by the Deutsche Forschungsgemeinschaft (DFG, German Research Foundation) under Germany´s Excellence Strategy – The Berlin Mathematics Research Center MATH+ (EXC-2046/1, project ID: 390685689)
\section{Preliminaries}\label{sec:flag=gr}
\subsection{Extended quiver}
In this subsection, we introduce the notion of extended quiver which allows to view partial flag varieties as quiver Grassmannians. Intuitively, a flag of quiver representations can be encoded as a subspace of a representation of the extended quiver.

\begin{defn}[Extended quiver]
For a quiver $Q$ and an integer $d \geqslant 1$, the extended quiver $Q_{d}$ is defined as follows:
\begin{itemize}
	\item The vertex set of $Q_d$ is defined as the Cartesian product of the vertex set of $Q$ and $\{1,\ldots,d\}$, i.e.,
	$$v(Q_d)=v(Q) \times \{1,\ldots,d\}.$$
	\item  There are two types of arrows: for each $(i,r) \in v(Q) \times \{1,\ldots,d-1\}$, there is one arrow from $(i,r)$ to $(i,r+1)$; for each arrow $i \longrightarrow j$ in $Q$ and $r \in \{1,\ldots,d\}$, there is one arrow from $(i,r)$ to $(j,r)$.
\end{itemize}
\end{defn}

The extended quiver $Q_d$ is exactly the same quiver as $\hat{\Gamma}_d$ in \cite[Definition 2.2]{maksimau2019flag}. The next definition is a small variation:

\begin{defn}[Strict extended quiver]
For a quiver $Q$ and an integer $d \geqslant 2$, the strict extended quiver $Q_{d,\operatorname{str}}$ is defined as follows:
\begin{itemize}
	\item The vertex set of $Q_d$ is defined as the Cartesian product of the vertex set of $Q$ and $\{1,\ldots,d\}$, i.e.,
	$$v(Q_{d,\operatorname{str}})=v(Q) \times \{1,\ldots,d\}.$$
	\item  We have two types of arrows: for each $(i,r) \in v(Q) \times \{1,\ldots,d-1\}$, there is one arrow from $(i,r)$ to $(i,r+1)$; for each arrow $i \longrightarrow j$ in $Q$ and $r \in \{2,\ldots,d\}$, there is one arrow from $(i,r)$ to $(j,r-1)$.
\end{itemize}
\end{defn}

\begin{eg}\label{eg:ext_quiver}
The (strict) extended quiver for a Dynkin quiver $Q$ of type $A_4$ looks as follows.
\[\begin{tikzcd}[column sep=3mm, row sep=5mm]
	&&&&&&&[5mm] \bullet && \bullet && \bullet && \bullet &[5mm] \bullet && \bullet && \bullet && \bullet \\
	&&&&&&& \bullet && \bullet && \bullet && \bullet & \bullet && \bullet && \bullet && \bullet \\
	\bullet && \bullet && \bullet && \bullet & \bullet && \bullet && \bullet && \bullet & \bullet && \bullet && \bullet && \bullet \\[-5mm]
	&&& {\hspace{-1cm}Q\hspace{-1cm}} &&&&&&& {\hspace{-1cm}Q_3\hspace{-1cm}} &&&&&&& {\hspace{-1cm}Q_{3,\operatorname{str}}\hspace{-1cm}}
	\arrow[from=3-1, to=3-3]
	\arrow[from=3-5, to=3-3]
	\arrow[from=3-5, to=3-7]
	\arrow[from=3-8, to=3-10]
	\arrow[from=3-12, to=3-10]
	\arrow[from=3-12, to=3-14]
	\arrow[from=2-8, to=2-10]
	\arrow[from=2-12, to=2-10]
	\arrow[from=2-12, to=2-14]
	\arrow[from=1-8, to=1-10]
	\arrow[from=1-12, to=1-10]
	\arrow[from=1-12, to=1-14]
	\arrow[from=2-8, to=1-8]
	\arrow[from=3-8, to=2-8]
	\arrow[from=3-10, to=2-10]
	\arrow[from=2-10, to=1-10]
	\arrow[from=2-12, to=1-12]
	\arrow[from=3-12, to=2-12]
	\arrow[from=3-14, to=2-14]
	\arrow[from=2-14, to=1-14]
	\arrow[from=2-15, to=3-17]
	\arrow[from=1-15, to=2-17]
	\arrow[from=2-19, to=3-17]
	\arrow[from=1-19, to=2-17]
	\arrow[from=1-19, to=2-21]
	\arrow[from=2-19, to=3-21]
	\arrow[from=2-21, to=1-21]
	\arrow[from=3-21, to=2-21]
	\arrow[from=3-19, to=2-19]
	\arrow[from=2-19, to=1-19]
	\arrow[from=3-17, to=2-17]
	\arrow[from=2-17, to=1-17]
	\arrow[from=2-15, to=1-15]
	\arrow[from=3-15, to=2-15]
\end{tikzcd}\]
\end{eg}

Next, we define the quiver algebras for later use.
\begin{defn}[Algebra of an extended quiver]\label{def:bqa}
For an extended quiver $Q_d$, let $KQ_d$ be the corresponding path algebra, and $I$ be the ideal of $KQ_d$ identifying all the paths with the same sources and targets. The algebra of the extended quiver $Q_d$ is defined as
$$R_d:= KQ_d/I.$$
Similarly, we define the algebra $R_{d,\operatorname{str}}:= KQ_{d,\operatorname{str}}/I$ for the strict extended quiver.
\end{defn}
By abuse of notation, we often abbreviate $R_d$ and $R_{d,\operatorname{str}}$ by $R$. 
\subsection{Canonical functor $\Phi$}
We follow \cite[2.3]{maksimau2019flag} in this subsection with a few variations. 
\begin{defn}[Partial flag]\label{def:pf}
For a quiver representation $X \in \rep(Q)$, a partial flag of $X$ is defined as an increasing sequence of subrepresentations of $X$. For an integer $d \geqslant 1$, we denote
$$\Flagd(X):=\left\{ 0 \subseteq M_1 \subseteq \cdots M_d \subseteq X \right\}$$
as the collection of all partial flags of length $d$, and call it the partial flag variety. 
\end{defn}
\begin{defn}[Strict partial flag]\label{def:spf}
For a quiver representation $X \in \rep(Q)$, a strict partial flag of $X$ is defined as an increasing sequence of subrepresentations $(M_k)_k$ of $X$ such that for any arrow $x \in v(Q)$ and any $k$, we have $x.M_{k+1} \subseteq M_{k}$. For an integer $d \geqslant 2$, we denote
$$\Flagdstr(X):=\left\{ 0 \subseteq M_1 \subseteq \cdots M_d \subseteq X\mid x.M_{k+1} \subseteq M_k \right\}$$
as the collection of all strict partial flags of length $d$, and call it the strict partial flag variety. 
\end{defn}
\begin{defn}[Grassmannian]
Let $R$ be the bounded quiver algebra defined in Definition \ref{def:pf} or \ref{def:spf}. For a module $T \in \Mod(R)$, the Grassmannian $\Gralg{R}(T)$ is defined as the set of all submodules of $T$, i.e.,
$$\Gralg{R}(T):=\{T' \subseteq T \text{ as the submodule}  \}.$$
\end{defn}
\begin{defn}[Canonical functor $\Phi$]
The canonical functor $\Phi:\rep(Q) \longrightarrow \Mod(R) $ is defined as follows:
\begin{itemize}
\item $\left(\Phi(X)\right)_{(i,r)}:=X_i$;
\item $\left(\Phi(X)\right)_{(i,r) \rightarrow (i,r+1)}:=\Id_{X_i}$;
\item Either $\left(\Phi(X)\right)_{(i,r) \rightarrow (j,r)}:=X_{i \rightarrow j}$ for $R=R_d$, \\or $\left(\Phi(X)\right)_{(i,r) \rightarrow (j,r-1)}:=X_{i \rightarrow j}$ for $R=R_{d,\operatorname{str}}$.
\end{itemize}
\end{defn}
The functor $\Phi$ helps to realize a partial flag as a quiver subrepresentation. 
\begin{proposition}\label{prop:flag=gr}
For a representation $X\in \rep(Q)$, the canonical functor $\Phi$ induces isomorphisms
$$\Flagd(X)\cong \Gralg{R_d}(\Phi(X)) \qquad \Flagdstr(X)\cong \Gralg{R_{d,\operatorname{str}}}(\Phi(X)).$$
\end{proposition}
\begin{proof}
The isomorphism maps a flag $M:M_1 \subseteq \cdots \subseteq M_d$ to a representation $\Phi'(M)$ with $\Phi'(M)_{(i,r)}=M_{i,r}$ and obvious morphisms for arrows. The non-strict case is mentioned in \cite[page 4]{maksimau2019flag} and the strict case works similarly.
\end{proof}
\begin{eg}\label{eg:bigpic}
Consider the quiver $Q\colon x \longrightarrow y \longleftarrow z \longrightarrow w$, and let $X\colon X_x \longrightarrow X_y \longleftarrow X_z \longrightarrow X_w$ be a representation. The varieties $\Flag{3}(X),\Flagstr{3}(X)$ then arise as quiver Grassmannian as shown in Figure \ref{fig:flagasgr}.
\begin{center}
	\begin{figure}[ht]
		\vspace{0cm}
		\centering
		\includegraphics[width=0.9\linewidth]{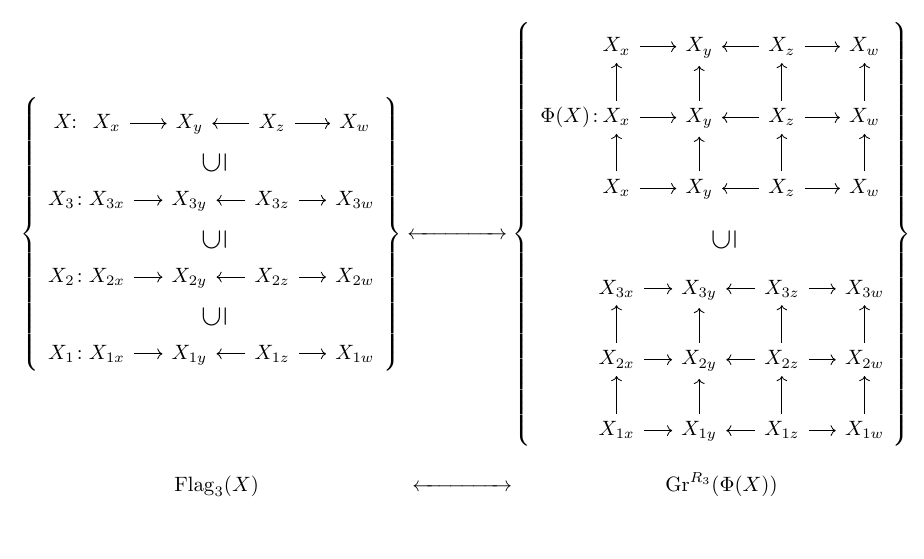}
		\vspace{-0.5cm}
		\includegraphics[width=\linewidth]{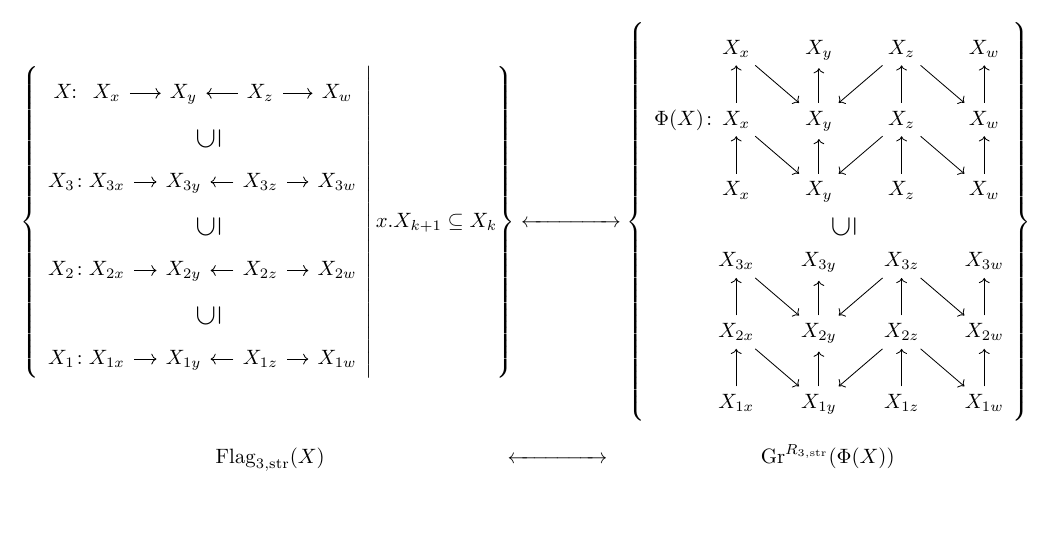}
		\vspace{-1cm}
		\caption{Quiver flag variety realized as quiver Grassmannian.}
		\vspace{0.3cm}
		\label{fig:flagasgr}
	\end{figure}
\end{center}
\end{eg}
In many cases, the proof of the strict case and the non-strict case is the same, so we often treat them in the same way. For example, we may abbreviate the formula in Proposition \ref{prop:flag=gr} as 
$$\Flag{}(X)\cong \Gr(\Phi(X)).$$
\subsection{Dimension vector}In this subsection we recall some notations of dimension vectors.
\begin{defn}[Dimension vector]
For a quiver $Q$ and a representation $M \in \rep(Q)$, the set of dimension vectors of $Q$ is defined as $\prod_{i \in v(Q)} \mathbb{Z}$, and the dimension vector of $M$ is defined as
$$\dimv M:=(\dim_K M_i)_{i \in v(Q)}.$$
\end{defn}
Moreover, if $R=KQ/I$ is a bounded quiver algebra, then every module $T \in \Mod(R)$ can be viewed as a representation of $Q$, so we automatically have a notion of dimension vector for $R$ and $T$.

Now we can write the (strict) partial flag variety and Grassmannian as disjoint union of several pieces. Since $v(Q_{d,(\operatorname{str})})=v(Q) \times \{1,\ldots,d\}$, any dimension vector $\dimvec{f}$ of $R$ can be viewed as $d$ dimension vectors $(\dimvec{f}_1,\ldots,\dimvec{f}_d)$. Define
\begin{equation*}
\begin{aligned}
  \Flag{\dimvec{f}}(X):=\;&\left\{ 0 \subseteq M_1 \subseteq \cdots M_d \subseteq X \;\middle|\; \dimv M_k=\dimvec{f}_k \right\} && \subseteq  \Flagd(X), \\ 
  \Flag{\dimvec{f},\operatorname{str}}(X):=\;&\left\{ 0 \subseteq M_1 \subseteq \cdots M_d \subseteq X \;\middle|\; x.M_{k+1} \subseteq M_k,\, \dimv M_k=\dimvec{f}_k \right\} && \subseteq  \Flagdstr(X), \\ 
  \Gralg{R}_{\dimvec{f}}(T):=\;&\{T' \subseteq T \text{ with } \dimv T'=\dimvec{f}\} && \subseteq  \Gralg{R}(T). 
\end{aligned}
\end{equation*}
Then from the Proposition \ref{prop:flag=gr} we get 
$$\Flag{\dimvec{f}}(X) \cong \Gralg{R_d}_{\dimvec{f}}(\Phi(X)) \qquad \Flag{\dimvec{f},\operatorname{str}}(X) \cong \Gralg{R_{d,\operatorname{str}}}_{\dimvec{f}}(\Phi(X)).$$
\begin{remark}
All the spaces we defined here have natural topologies and variety structures. For example, by the standard embedding
\[\begin{tikzcd}
	\Gralg{R}_{\dimvec{f}}(T) & \displaystyle\prod_{(i,r) \in v\left(Q_{d,(str)}\right)} \Gr_{\dimvec{f}_{i,r}}\big(T_{(i,r)}\big),
	\arrow[hook, from=1-1, to=1-2]
\end{tikzcd}\]
$\Gralg{R}_{\dimvec{f}}(T)$ is then endowed with the subspace topology and subvariety structure.
\end{remark}

Finally, we need to define the Euler form of two dimension vectors. For this we need to define the set of virtual arrows of the quivers $Q_d$ and $Q_{d,\operatorname{str}}$. Following Example \ref{example:virtualarrow}, the virtual arrows of the quivers $Q_3$ and $Q_{3,\operatorname{str}}$ are depicted in red.
\begin{defn}[Virtual arrows of the quiver $Q_d$]
For $d \geqslant 1$, the virtual arrows of the quiver $Q_d$ are defined as a triple $\big(va(Q_d),s,t\big)$, where
$$va(Q_d):=a(Q) \times \{1,\ldots, d-1 \}$$
is a finite set, and $s,t: va(Q_d) \longrightarrow v(Q_d)$ are maps defined by
$$s\big((i\rightarrow j,r)\big)=(i,r) \qquad t \big((i\rightarrow j,r)\big)=(j,r+1).$$
\end{defn}
\begin{defn}[Virtual arrows of the quiver $Q_{d,\operatorname{str}}$]
For $d \geqslant 2$, the virtual arrows of the quiver $Q_{d,\operatorname{str}}$ is defined as a triple $\big(va(Q_{d,\operatorname{str}}),s,t\big)$, where
$$va(Q_{d,\operatorname{str}}):=a(Q) \times \{2,\ldots, d-1 \}$$
is a finite set, and $s,t: va(Q_{d,\operatorname{str}}) \longrightarrow v(Q_{d,\operatorname{str}})$ are maps defined by
$$s\big((i\rightarrow j,r)\big)=(i,r) \qquad t \big((i\rightarrow j,r)\big)=(j,r).$$
\end{defn}
\begin{eg}\label{example:virtualarrow}
$\qquad$
	\begin{figure}[ht]
		\vspace{0cm}
		\centering
\[\begin{tikzcd}[column sep=3mm, row sep=5mm]
	&&&&&&&[5mm] \bullet && \bullet && \bullet && \bullet &[5mm] \bullet && \bullet && \bullet && \bullet \\
	&&&&&&& \bullet && \bullet && \bullet && \bullet & \bullet && \bullet && \bullet && \bullet \\
	\bullet && \bullet && \bullet && \bullet & \bullet && \bullet && \bullet && \bullet & \bullet && \bullet && \bullet && \bullet \\[-5mm]
	&&& {\hspace{-1cm}Q\hspace{-1cm}} &&&&&&& {\hspace{-1cm}Q_3\hspace{-1cm}} &&&&&&& {\hspace{-1cm}Q_{3,\operatorname{str}}\hspace{-1cm}}
	\arrow[from=3-1, to=3-3]
	\arrow[from=3-5, to=3-3]
	\arrow[from=3-5, to=3-7]
	\arrow[from=3-8, to=3-10]
	\arrow[from=3-12, to=3-10]
	\arrow[from=3-12, to=3-14]
	\arrow[from=2-8, to=2-10]
	\arrow[from=2-12, to=2-10]
	\arrow[from=2-12, to=2-14]
	\arrow[from=1-8, to=1-10]
	\arrow[from=1-12, to=1-10]
	\arrow[from=1-12, to=1-14]
	\arrow[from=2-8, to=1-8]
	\arrow[from=3-8, to=2-8]
	\arrow[from=3-10, to=2-10]
	\arrow[from=2-10, to=1-10]
	\arrow[from=2-12, to=1-12]
	\arrow[from=3-12, to=2-12]
	\arrow[from=3-14, to=2-14]
	\arrow[from=2-14, to=1-14]
	\arrow[from=2-15, to=3-17]
	\arrow[from=1-15, to=2-17]
	\arrow[from=2-19, to=3-17]
	\arrow[from=1-19, to=2-17]
	\arrow[from=1-19, to=2-21]
	\arrow[from=2-19, to=3-21]
	\arrow[from=2-21, to=1-21]
	\arrow[from=3-21, to=2-21]
	\arrow[from=3-19, to=2-19]
	\arrow[from=2-19, to=1-19]
	\arrow[from=3-17, to=2-17]
	\arrow[from=2-17, to=1-17]
	\arrow[from=2-15, to=1-15]
	\arrow[from=3-15, to=2-15]
	\arrow[dashed,color={rgb,255:red,255;green,51;blue,51}, from=2-8, to=1-10]
	\arrow[dashed,color={rgb,255:red,255;green,51;blue,51}, from=3-8, to=2-10]
	\arrow[dashed,color={rgb,255:red,255;green,51;blue,51}, from=2-12, to=1-10]
	\arrow[dashed,color={rgb,255:red,255;green,51;blue,51}, from=3-12, to=2-10]
	\arrow[dashed,color={rgb,255:red,255;green,51;blue,51}, from=3-12, to=2-14]
	\arrow[dashed,color={rgb,255:red,255;green,51;blue,51}, from=2-12, to=1-14]
	\arrow[dashed,color={rgb,255:red,255;green,51;blue,51}, from=2-15, to=2-17]
	\arrow[dashed,color={rgb,255:red,255;green,51;blue,51}, from=2-19, to=2-17]
	\arrow[dashed,color={rgb,255:red,255;green,51;blue,51}, from=2-19, to=2-21]
\end{tikzcd}\]
\vspace{-5mm}
		\label{fig:virtualarrow}
		\vspace{2mm}
	\end{figure}
\end{eg}
\begin{defn}[Euler form of $R$]
Let $R$ be a bounded quiver algebra defined in Definition \ref{def:bqa}. We denote
\begin{equation*}
\begin{aligned}
	v(R):&= \{\text{vertices in $Q_d$ or $Q_{d,\operatorname{str}}$}\}, \\
	a(R):&= \{\text{arrows in $Q_d$ or $Q_{d,\operatorname{str}}$}\}, \\
	va(R):&= \{\text{virtual arrows in $Q_d$ or $Q_{d,\operatorname{str}}$}\}. \\
\end{aligned}
\end{equation*}

For two dimension vectors $\dimvec{f},\dimvec{g}$ of $R$, the Euler form $\left< \dimvec{f},\dimvec{g}\right>_R$ is defined by
	$$\left< \dimvec{f},\dimvec{g}\right>_R:= \sum_{i \in v(R)} f_ig_i - \sum_{b \in a(R)} f_{s(b)}g_{t(b)}+ \sum_{c \in va(R)} f_{s(c)}g_{t(c)}.$$
\end{defn}
\subsection{Ext-vanishing properties}
We will show that some higher rank extension group are zero, which will be a key ingredient in the proofs of the next section.

For a bounded quiver algebra $R$ defined in Definition \ref{def:bqa}, we have a standard resolution for every $R$-module $T$: 
\[
\begin{tikzcd}[row sep=0mm,column sep=3mm]
	0 & {\displaystyle\bigoplus_{c \in va(Q)} \hspace{-3mm}Re_{t(c)} \otimes_K e_{s(c)}T} & {\displaystyle\bigoplus_{b \in a(Q)}\hspace{-2mm} Re_{t(b)} \otimes_K e_{s(b)}T} & {\displaystyle\bigoplus_{i \in v(Q)}\hspace{-2mm} Re_{i} \otimes_K e_{i}T} & T & 0 \\
	& {\hspace{10mm} r \otimes x} & {\substack{\phantom{+}rc_1 \otimes x +r \otimes b_1x\\-rc_2 \otimes x -r \otimes b_2x}\hspace{-5mm}} & {\hspace{6mm}r \otimes x} & rx \\
	&& {\hspace{6mm}r \otimes x} & {rb \otimes x-r \otimes bx\hspace{-3mm}}
	\arrow[from=1-1, to=1-2]
	\arrow[from=1-2, to=1-3]
	\arrow[from=1-3, to=1-4]
	\arrow[from=1-4, to=1-5]
	\arrow[from=1-5, to=1-6]
	\arrow[maps to, from=2-2, to=2-3]
	\arrow[maps to, from=2-4, to=2-5]
	\arrow[maps to, from=3-3, to=3-4]
\end{tikzcd}
\]
There are exactly two paths of length two from $s(c)$ to $t(c)$ for any virtual arrow $c$, which we denoted by $b_1c_1$ and $b_2c_2$ in the above. By definition, these paths are identified in $R$.

\begin{lemma}\label{lm:Extvan}
Let $M,N \in \rep(Q)$.
\begin{enumerate}[(1)]
	\item $\gldim R \leqslant 2$;\label{lm:gldim}
	\item The functor $\Phi:\rep(Q) \longrightarrow \Mod(R)$ is exact and fully faithful;\label{lm:functorisexact}
	\item $\Phi$ maps projective module to projective module, and maps injective module to injective module;\label{lm:toproj}
	\item $\Ext^i_{KQ}(M,N) \cong \Ext^i_{R}(\Phi(M),\Phi(N))$;\label{lm:isoofext}
	\item $\projdim \Phi(M) \leqslant 1. \injdim \Phi(M) \leqslant 1$;\label{lm:projdim}
\end{enumerate}
\end{lemma}
\begin{proof}$\,$

For (\ref{lm:gldim}), this follows from the standard resolution.

For (\ref{lm:functorisexact}), it follows by direct inspection, see \cite[Lemma 2.3]{maksimau2019flag}.

For (\ref{lm:toproj}), we reduce to the case of indecomposable projective modules, and observe that $$\Phi(P(i))=P\big((i,1)\big),\qquad\Phi(I(i))=I\big((i,d)\big).$$

For (\ref{lm:isoofext}), it comes from the fact that $\Phi$ is fully faithful and maps projective module to projective module.
%

For (\ref{lm:projdim}), notice that the minimal projective resolution of $M$ is of length 1, and $\Phi(-)$ sends the projective resolution of $M$ to the projective resolution of $\Phi(M)$ by  (\ref{lm:toproj}), thus we get $\projdim \Phi(M) \leqslant 1$. The injective dimension of $\Phi(M)$ is computed in a similar way.
\end{proof}
The following key lemma will be crucial later.
\begin{lemma}\label{lm:Ext2van}
Let $X,S \in \rep(Q)$ and $V \subseteq \Phi(X), W \subseteq \Phi(S)$, $T \in \Mod(R)$. Then $\Ext^2_{R}(W,T)=0$ and $\Ext^2_{R}(T,\Phi(X)/V)=0$.
\end{lemma}
\begin{proof}
The short exact sequence 
$$0 \longrightarrow W \longrightarrow \Phi(S) \longrightarrow \Phi(S)/W \longrightarrow 0$$
induces the long exact sequence 
$$\cdots \longrightarrow \Ext^2_{R}(\Phi(S),T) \longrightarrow \Ext^2_{R}(W,T) \longrightarrow \Ext^3_{R}(\Phi(S)/W,T) \longrightarrow \cdots.$$
By Lemma \ref{lm:Extvan} (\ref{lm:gldim}) and (\ref{lm:projdim}), $\Ext^3_{R}(\Phi(S)/W,T)$ and $\Ext^2_{R}(\Phi(S),T)$ are both $0$, so $\Ext^2_{R}(W,T)=0$.

Similarly, from the short exact sequence
$$0 \longrightarrow V \longrightarrow \Phi(X) \longrightarrow \Phi(X)/V \longrightarrow 0$$
we get the induced long exact sequence
$$\cdots \longrightarrow \Ext^2_{R}(T,\Phi(X)) \longrightarrow \Ext^2_{R}(T,\Phi(X)/V) \longrightarrow \Ext^3_{R}(T,V) \longrightarrow \cdots,$$
so $\Ext^2_{R}(T,\Phi(X)/V)=0$.
\end{proof}
We will frequently use extension groups as well as long exact sequences, so we introduce some abbreviations. For $Q$-representations $M,N$ and $R$-modules $T,T'$, we denote  
\begin{equation*}
\begin{aligned}[]
	[M,N]^i:&=\dim_K \Ext^i_{KQ} (M,N)\qquad [M,N]:=\dim_K \Hom_{KQ} (M,N)\\
	[T,T']^i:&=\dim_K \Ext^i_R (T,T')\qquad\hspace{0.5cm} [T,T']:=\dim_K \Hom_R (T,T')
\end{aligned}
\end{equation*}
and write the Euler form as
$$\left< T,T'\right>_R:= \sum_{i=0}^{\infty} (-1)^i [T,T']^i \quad=[T,T']-[T,T']^1+[T,T']^2.$$
\begin{lemma}[Homological interpretation of the Euler form]
	For two $R$-modules $T,T'$, we have
	$$	\left< T,T'\right>_R = 	\left< \dimv T,\dimv T'\right>_R.$$
\end{lemma}
\begin{proof}
Compute $\left< T,T'\right>_R$ by applying the functor $\Hom_R(-,T')$ to the standard resolution of the $R$-module $T$.
\end{proof}
%


\section{Main Theorem}\label{sec:mainthm}
In this section we state and prove the main theorems, which are essential in Section \ref{sec:Dynkin} and \ref{sec:affine}.
   
Let $\eta: 0\longrightarrow X \stackrel{\iota}{\longrightarrow} Y \stackrel{\pi}{\longrightarrow} S \longrightarrow 0$ be a short exact sequence in $\rep(Q)$. Consider the canonical \textbf{non-continuous} map
$$\Psi: \Grr(\Phi(Y)) \longrightarrow \Grr(\Phi(X)) \times \Grr(\Phi(S)) \qquad U \longmapsto \left([\Phi(\iota)]^{-1}(U),[\Phi(\pi)](U)   \right).$$
Denote the set
$$\Grr(\Phi(Y))_{\dimvec{f},\dimvec{g}}:= \Psi^{-1}\Big(\Grr_{\dimvec{f}}(\Phi(X)) \times \Grr_{\dimvec{g}}(\Phi(S))\Big)$$
and
let $\Psi_{\dimvec{f},\dimvec{g}}$ be the map $\Psi$ restricted to $\Grr(\Phi(Y))_{\dimvec{f},\dimvec{g}}$, i.e.,
$$\Psi_{\dimvec{f},\dimvec{g}}: \Grr(\Phi(Y))_{\dimvec{f},\dimvec{g}} \longrightarrow \Grr_{\dimvec{f}}(\Phi(X)) \times \Grr_{\dimvec{g}}(\Phi(S)).$$
\begin{remark}\label{rem:topo}
Even though $\Psi$ is not continuous, $\Psi_{\dimvec{f},\dimvec{g}}$ is continuous. Moreover, for any dimension vectors $\dimvec{f},\dimvec{g}$, the set
$$\Gr(\Phi(Y))_{\geqslant \dimvec{f}, \leqslant \dimvec{g} }:= \left\{ U \in \Gr(\Phi(Y)) \,\middle|\,\begin{aligned}
\dimv [\Phi(\iota)]^{-1}(U) &\geqslant \dimvec{f}\\ \dimv [\Phi(\pi)]\phantom{^{0}}(U) &\leqslant \dimvec{g}
\end{aligned}  \right\}$$
is closed in $\Gr(\Phi(Y))$. This gives us a filtration 
$$0= Z_0 \subset Z_1 \subset \cdots \subset Z_d=\Gr_{\dimvec{h}}(\Phi(Y))$$
with $Z_i$ closed and $Z_{i+1} \setminus Z_i$ isomorphic to $\Gr(\Phi(Y))_{ \dimvec{f},\dimvec{g}}$ for some $\dimvec{f},\dimvec{g}$. Therefore, from the affine pavings of $\Gr(\Phi(Y))_{ \dimvec{f},\dimvec{g}}$ (for every $\dimvec{f},\dimvec{g}$) one can construct one affine paving of $\Gr_{\dimvec{h}}(\Phi(Y))$.
\end{remark} 
\begin{theorem}\label{thm:main1}
	If $\eta$ splits, then $\Psi$ is surjective. Moreover, if $[S,X]^1=0$, then $\Psi_{\dimvec{f},\dimvec{g}}$ is a Zariski-locally trivial affine bundle of rank $\left< \dimvec{g},\dimv \Phi(X) - \dimvec{f}\right>_R$.
\end{theorem}
\begin{theorem}[Generalizes {\cite[Theorem 32]{irelli2019cell}}]\label{thm:main2}
	When $\eta$ does not split and $[S,X]^1=1$, 
	$$\Img \Psi_{\dimvec{f},\dimvec{g}} = \bigg(\Grr_{\dimvec{f}}(\Phi(X)) \times \Grr_{\dimvec{g}}(\Phi(S)) \bigg) \setminus \bigg(\Grr_{\dimvec{f}}(\Phi(X_S)) \times \Grr_{\dimvec{g}-\dimv \Phi(S^X)}\left(\Phi(S/S^X)\right) \bigg)$$
	where 
	\begin{equation*}
	\begin{aligned}
	X_S:&= \max \left\{ M \subseteq X \,\middle|\; [S,X/M ]^1=1 \right\} \subseteq X,\\
	S^X:&= \max \left\{ M \subseteq S \,\middle|\; [M,X]^1=1 \right\} \subseteq S.
	\end{aligned}
	\end{equation*}
Moreover, $\Psi_{\dimvec{f},\dimvec{g}}$ is a Zarisky-locally trivial affine bundle of rank $\left< \dimvec{g},\dimv \Phi(X) - \dimvec{f}\right>_R$ over $\Img \Psi_{\dimvec{f},\dimvec{g}}$.
\end{theorem}

We will spend the rest of the section proving these theorems. We investigate the image as well as the fiber of $\Psi$ respectively.
\begin{lemma}[Follows {\cite[Lemma 21]{irelli2019cell}}]\label{lm:split}
The element $(V,W) \in \Grr(\Phi(X)) \times \Grr(\Phi(S))$ lies in the image of $\Psi$ if and only if the canonical map $\Ext^1(\Phi(S),\Phi(X)) \longrightarrow \Ext^1(W,\Phi(X)/V)$ maps $\eta$ to 0.
\end{lemma}
\begin{proof}
	The canonical map is defined as follows:
\[\begin{tikzcd}
	\eta &[-30pt] \in &[-30pt] {\Ext^1(\Phi(S),\Phi(X))} & 0 & {\Phi(X)} & {\Phi(Y)} & {\Phi(S)} & 0 \\
	&& {\Ext^1(W,\Phi(X))} & 0 & {\Phi(X)} & {\pi^{-1}(W)} & W & 0 \\
	{\bar{\eta}} & \in & {\Ext^1(W,\Phi(X)/V)} & 0 & {\Phi(X)/V} & {\pi^{-1}(W)/V} & W & 0
	\arrow[from=2-5, to=3-5]
	\arrow[Rightarrow, no head, from=2-5, to=1-5]
	\arrow[from=2-6, to=1-6]
	\arrow[from=2-6, to=3-6]
	\arrow[Rightarrow, no head, from=2-7, to=3-7]
	\arrow[from=2-7, to=1-7]
	\arrow[from=1-4, to=1-5]
	\arrow[from=1-5, to=1-6]
	\arrow["\Phi(\pi)", from=1-6, to=1-7]
	\arrow[from=1-7, to=1-8]
	\arrow[from=2-4, to=2-5]
	\arrow[from=2-5, to=2-6]
	\arrow[from=2-6, to=2-7]
	\arrow[from=2-7, to=2-8]
	\arrow[from=3-4, to=3-5]
	\arrow[from=3-5, to=3-6]
	\arrow[from=3-6, to=3-7]
	\arrow[from=3-7, to=3-8]
	\arrow[from=1-3, to=2-3]
	\arrow[from=2-3, to=3-3]
	\arrow[maps to, from=1-1, to=3-1]
\end{tikzcd}\]
so $\bar{\eta}=0$ if and only if the last short exact sequence splits, that means, there exists a submodule $U \subseteq \Phi(Y)$, such that $\Phi(\pi)(U)=W$ and $U \cap \Phi(X) =V$.
\end{proof}

\begin{corollary}\label{cor:img1}
	Resume the notations of Lemma \ref{lm:split} When $\eta$ splits, then $\Psi$ is surjective.
\end{corollary}
\begin{lemma}
	The canonical map $\Ext^1(\Phi(S),\Phi(X)) \longrightarrow \Ext^1(W,\Phi(X)/V)$ is surjective.
\end{lemma}
\begin{proof}
	By using the long exact sequence of extension groups and the fact that $\allowbreak\Ext^2(\Phi(S)/W,\Phi(X))=0$ and $\Ext^2(W,V)=0$ by Lemma \ref{lm:Ext2van}, the maps
	$$\Ext^1(\Phi(S),\Phi(X)) \longrightarrow \Ext^1(W,\Phi(X))\qquad \Ext^1(W,\Phi(X)) \longrightarrow \Ext^1(W,\Phi(X)/V)$$
	are both surjective. Thus the composition is also surjective.
\end{proof}
\begin{corollary}[]\label{cor:0or1}
	Let $W \subseteq \Phi(S), V \subseteq \Phi(X)$ be $R$-submodules, then
	$$[W,\Phi(X)/V]^1 \leqslant [\Phi(S),\Phi(X)]^1=[S,X]^1.$$
	In particular, when $[S,X]^1=1$, we get $[W,\Phi(X)/V]^1=0 \text{ or }1$; when $\eta$ generates $\Ext^1(S,X)$, we get
	$$(V,W) \in \Img \Psi \iff [W,\Phi(X)/V]^1=0.$$	
\end{corollary}
 	
In the case where $\eta$ generates $\Ext^1(S,X)$, we want to describe $\Img \Psi$ more precisely. For this reason we need to introduce two new $R$-modules:
	\begin{equation*}
	\begin{aligned}
	\widetilde{X_S}:&= \max \left\{ V \subseteq \Phi(X) \,\middle|\; [\Phi(S),\Phi(X)/V ]^1=1 \right\} \subseteq \Phi(X),\\
	\widetilde{S^X}:&= \max \left\{ W \subseteq \Phi(S) \,\middle|\; [W,\Phi(X)]^1=1 \right\} \subseteq \Phi(S).
	\end{aligned}
	\end{equation*}
$\widetilde{X_S}$ and $\widetilde{S^X}$ are well-defined because of the following lemma:
\begin{lemma}[Follows {\cite[Lemma 27]{irelli2019cell}}]\
\begin{enumerate}[(i)] 
	\item Let $V,V' \subset \Phi(X)$ such that $[\Phi(S),\Phi(X)/V ]^1=[\Phi(S),\Phi(X)/V']^1=1 .$ Then $\allowbreak[\Phi(S),\Phi(X)/(V+V') ]^1=1$.
	\item  Let $W,W' \subset \Phi(S)$ such that $[W,\Phi(X)]^1=[W',\Phi(X)]^1=1 .$ Then $\allowbreak[W\cap W',\Phi(X)]^1=1$.
\end{enumerate}
\end{lemma}
\begin{proof}
We only prove (i). (ii) is similar.

From the short exact sequence 
$$0 \longrightarrow \Phi(X)/(V\cap V') \longrightarrow \Phi(X)/V \oplus \Phi(X)/V' \longrightarrow \Phi(X)/(V+V') \longrightarrow 0,$$
we get the long exact sequence
$$\hspace{-0.4cm}\cdots\rightarrow \Ext^1\!\!\left(\Phi(S),\textstyle\frac{\Phi(X)}{V\cap V'}\right) \rightarrow \Ext^1\!\!\left(\Phi(S),\textstyle\frac{\Phi(X)}{V}\right) \oplus \;\Ext^1\!\!\left(\Phi(S),\textstyle\frac{\Phi(X)}{V'}\right) \rightarrow \Ext^1\!\!\left(\Phi(S),\textstyle\frac{\Phi(X)}{V+ V'}\right) \rightarrow\cdots.$$
By Corollary \ref{cor:0or1}, $[\Phi(S),\Phi(X)/(V\cap V')]^1\leqslant 1, \; [\Phi(S),\Phi(X)/(V+V')]^1\leqslant 1$, and this forces $[\Phi(S),\Phi(X)/(V+V')]^1= 1$.
\end{proof}

\begin{lemma}[Follows {\cite[Lemma 31(1)(2)]{irelli2019cell}}, with the same proof]\label{lemma:second_description_XS}
Let $\tau$ be the Auslander--Reiten translation.\\
Let $f:X \longrightarrow \tau S$ be a non-zero morphism,\footnote{Since $X$ is not injective, $[X,\tau S]=[S,X]^1=1$, $f$ is uniquely determined up to a constant.} then $X_S=\ker (f)$;\\
 also, $\Phi(f): \Phi(X) \longrightarrow \Phi(\tau S)$ is a non-zero morphism, $\widetilde{X_S}=\ker (\Phi(f))$.
\end{lemma}

\begin{proof}
For any $M \subseteq X$, we have
\begin{equation*}
\begin{aligned}
  \Ext^1(S,X/M)^{\vee}\cong\;& \Homup(X/M,\tau S) \\ 
  \cong\;& \left\{ g \in \Hom(X,\tau S) \middle| \,g|_M=0 \right\} \\ 
  \cong\;& \begin{cases}
  \mathbb{C}, & M \subseteq \ker f\\
  0, & M \nsubseteq \ker f,
  \end{cases} \\ 
\end{aligned}
\end{equation*}
so $[S,X/M]^1=1$ exactly when $M \subseteq \ker f$. Thus $X_S=\ker f$.

For $\Phi(f)$ it is similar. For any $V \subseteq \Phi(X)$, we have
\begin{equation*}
\begin{aligned}
  \Ext^1(\Phi(S),\Phi(X)/V)^{\vee}\cong\;& \Homup(\Phi(X)/V,\tau \Phi(S)) \\ 
  \cong\;& \Homup(\Phi(X)/V,\Phi(\tau S)) \\   
  \cong\;& \left\{ g \in \Hom(\Phi(X), \Phi(\tau S)) \middle| \,g|_V=0 \right\} \\ 
  \cong\;& \begin{cases}
  \mathbb{C}, & V \subseteq \ker \Phi(f)\\
  0, & V \nsubseteq \ker \Phi(f),
  \end{cases} \\ 
\end{aligned}
\end{equation*}
so $[\Phi(S),\Phi(X)/V]^1=1$ exactly when $V \subseteq \ker \Phi(f)$. Thus $\widetilde{X_S}=\ker (\Phi(f))$.
\end{proof}
\begin{corollary}
	$\widetilde{X_S}=\Phi(X_S)$\textcolor{black}{(since $\widetilde{X_S}=\ker (\Phi(f))=\Phi(\ker(f))=\Phi(X_S)$)}.
\end{corollary}
By a dual argument, one can show that $\widetilde{S^X}=\Phi(S^X)$.
\begin{lemma}[Follows {\cite[Lemma 31(6)]{irelli2019cell}}]
For $V \subseteq \Phi(X)$ and $W \subseteq \Phi(S)$, we have 
$$[W,\Phi(X)/V]^1=0 \iff V \nsubseteq \Phi(X_S) \text{ or }W \nsupseteq \Phi(S^X).$$
\end{lemma}
\begin{proof}
$\Leftarrow$: Without loss of generality suppose $V \nsubseteq \Phi(X_S)$, then
$$V \nsubseteq \Phi(X_S) \iff [\Phi(S),\Phi(X)/V ]^1=0 \Rightarrow  [W,\Phi(X)/V ]^1=0.$$

\hspace{0.8cm}$\Rightarrow$: If not, then $V \subseteq \Phi(X_S) \text{ and }W \supseteq \Phi(S^X)$, and\footnote{$[S^X,X/X_S]^1=1$ follows from \cite[Lemma 31(5)]{irelli2019cell}.} 
\begin{align*}
        &[W,\Phi(X)/V]^1 \geqslant [\Phi(S^X),\Phi(X)/\Phi(X_S)]^1 =[S^X,X/X_S]^1=1. \qedhere
\end{align*}
\end{proof}
\begin{corollary}\label{cor:img2}
When $\eta$ generates $\Ext^1(S,X)$, we have 
	$$\Img \Psi_{\dimvec{f},\dimvec{g}} = \bigg(\Grr_{\dimvec{f}}(\Phi(X)) \times \Grr_{\dimvec{g}}(\Phi(S)) \bigg) \setminus \bigg(\Grr_{\dimvec{f}}(\Phi(X_S)) \times \Grr_{\dimvec{g}-\dimv \Phi(S^X)}\left(\Phi(S/S^X)\right) \bigg).$$
\end{corollary}

\begin{lemma}\label{lem:torsor}
	For $(V,W) \in \Img \Psi$, the preimage of $(V,W)$ is a torsor of $\,\Hom_{R}(W,\Phi(X)/V)$. Hence, there is a non-canonical isomorphism
	$$\Psi^{-1}((V,W)) \cong \Hom_{R}(W,\Phi(X)/V).$$
\end{lemma}
\begin{proof}Recall the commutative diagram
\[\begin{tikzcd}
	\eta &[-30pt] \in & [-30pt]{\Ext^1(\Phi(S),\Phi(X))} & 0 & {\Phi(X)} & {\Phi(Y)} & {\Phi(S)} & 0 \\
	&& {\Ext^1(W,\Phi(X))} & 0 & {\Phi(X)} & {\pi^{-1}(W)} & W & 0 \\
	{\bar{\eta}} & \in & {\Ext^1(W,\Phi(X)/V)} & 0 & {\Phi(X)/V} & {\pi^{-1}(W)/V} & W & 0
	\arrow[from=2-5, to=3-5]
	\arrow[Rightarrow, no head, from=2-5, to=1-5]
	\arrow[from=2-6, to=1-6]
	\arrow[from=2-6, to=3-6]
	\arrow[Rightarrow, no head, from=2-7, to=3-7]
	\arrow[from=2-7, to=1-7]
	\arrow[from=1-4, to=1-5]
	\arrow[from=1-5, to=1-6]
	\arrow["\Phi(\pi)", from=1-6, to=1-7]
	\arrow[from=1-7, to=1-8]
	\arrow[from=2-4, to=2-5]
	\arrow[from=2-5, to=2-6]
	\arrow[from=2-6, to=2-7]
	\arrow[from=2-7, to=2-8]
	\arrow[from=3-4, to=3-5]
	\arrow["\iota", from=3-5, to=3-6]
	\arrow["{\pi'}", from=3-6, to=3-7]
	\arrow[from=3-7, to=3-8]
	\arrow[from=1-3, to=2-3]
	\arrow[from=2-3, to=3-3]
	\arrow[maps to, from=1-1, to=3-1]
	\arrow["\theta"', color={rgb,255:red,255;green,51;blue,58}, curve={height=-18pt}, from=3-7, to=3-6]
\end{tikzcd}\]
When $(V,W) \in \Img \Psi$, $\bar{\eta}$ is split, and each split morphism $\theta$ gives us an element in $\Psi^{-1}((V,W))$. If we fix one split morphism $\theta_0$, then the other split morphisms are all of the form $\theta_0 + \textcolor{black}{\iota \circ} f$ where $f \in \Hom_{R}(W,\Phi(X)/V)$(and this form is unique). So
\belowdisplayskip=-12pt
$$\Psi^{-1}((V,W)) \cong \{ \theta: \text{ split morphism} \} \cong \Hom_{R}(W,\Phi(X)/V).$$
\end{proof}
\begin{remark}\label{rem:bundleprop}
Any point $(V,W) \in \Img \Psi_{\dimvec{f},\dimvec{g}}$ can be also viewed as a morphism
$$f: \Spec K \longrightarrow \Img \Psi_{\dimvec{f},\dimvec{g}}  \subseteq \Grr_{\dimvec{f}}(\Phi(X)) \shorttimes \Grr_{\dimvec{g}}(\Phi(S)) $$
where Grassmannian are viewed as moduli spaces over $K$. Essentially by replacing $\Spec K$ by any locally closed reduced subscheme $\Spec A$ of $\Img \Psi_{\dimvec{f},\dimvec{g}}$ in Lemma \ref{lem:torsor}, we can run the machinery of algebraic geometry, and mimic the proof of \cite[Theorem 24]{irelli2019cell} to show that $\Psi_{\dimvec{f},\dimvec{g}}$ is a Zariski-locally trivial affine bundle over $\Img \Psi_{\dimvec{f},\dimvec{g}}$ when $\eta$ generates $\Ext^1(S,X)$. Roughly, there are 4 steps:
\begin{enumerate}[1.]
\item Realise Grassmannians as representable functors, and replace $K$-modules by $A$-modules;
\item Verify that $\Psi_{\dimvec{f},\dimvec{g}}^{-1} (\Spec A)$ is a $\Hom_A(\mathcal{W},\Phi(X)_A/\mathcal{V})$-torsor, where $$(\mathcal{V},\mathcal{W}) \in \Grr_{\dimvec{f}}(\Phi(X))(A) \shorttimes \Grr_{\dimvec{g}}(\Phi(S))(A)$$ corresponds to the immersion $\Spec A \hookrightarrow \Img \Psi_{\dimvec{f},\dimvec{g}}\,$;
\item Verify that $\Hom_A(\mathcal{W},\Phi(X)_A/\mathcal{V})$ is a vector bundle over $\Spec A$ of constant dimension $\left< \dimvec{f},\dimv \Phi(X)-\dimvec{g}\right>_R\,$;
\item Find a section of $\Psi_{\dimvec{f},\dimvec{g}}^{-1} (\Spec A) \longrightarrow \Spec A$, which is essentially the splitting $\theta$ in \cite[Lemma 22]{irelli2019cell}.
\end{enumerate}

\end{remark}
\begin{proof}[{Proof of Theorem \ref{thm:main1} and \ref{thm:main2}}]
We have already computed $\Img \Psi$ in Corollary \ref{cor:img1} and \ref{cor:img2}. In both cases $\eta$ generates $\Ext^1(S,X)$, so by Corollary \ref{cor:0or1} we get
\begin{equation*}
\begin{aligned}
(V,W) \in \Img \Psi_{\dimvec{f},\dimvec{g}} &\Longleftrightarrow [W,\Phi(X)/V]^1=0\\
& \Longrightarrow [W,\Phi(X)/V]=\left< W,\Phi(X)/V\right>_R=\left< \dimvec{f},\dimv \Phi(X)-\dimvec{g}\right>_R.
\end{aligned}
\end{equation*}
From Remark \ref{rem:bundleprop}, $\Psi_{\dimvec{f},\dimvec{g}}$ is a Zariski-locally trivial affine bundle.
\end{proof}

\section{Application: Dynkin Case}\label{sec:Dynkin}

In this section and the next, the proof of Theorem 4.1 is the main subject, with $Q$ representing a Dynkin quiver throughout.

\begin{theorem}\label{thm:Dynkincase}
For any Dynkin quiver $Q$ and any representation $M \in \rep(Q)$, the (strict) partial flag variety $\Flag{}(M)\cong\Grr(\Phi(M))$ has an affine paving.
\end{theorem}

Before discussing the proof of the affine paving property, we introduce some numerical concepts, which can be seen as a measure of the ``complexity" of the representation.

\captionsetup[subfigure]{labelformat=brace}
\begin{center}
\begin{figure}[ht]
\subcaptionbox[$E_6$]{$E_6$}[50mm][c]
{
\centering
\vspace{0mm}
\includegraphics[scale=0.8]{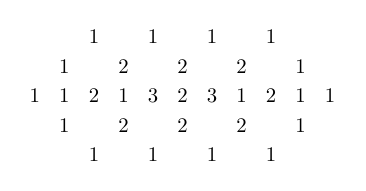}  
\vspace{1mm}
}~
\subcaptionbox[$E_7$]{$E_7$}[80mm][c]
{
\centering
\vspace{0mm}
\includegraphics[scale=0.8]{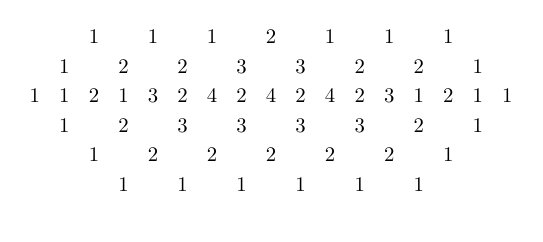}  
}
\subcaptionbox[$E_8$]{$E_8$}[130mm][c]
{
\centering
\vspace{0mm}
\includegraphics[scale=0.8]{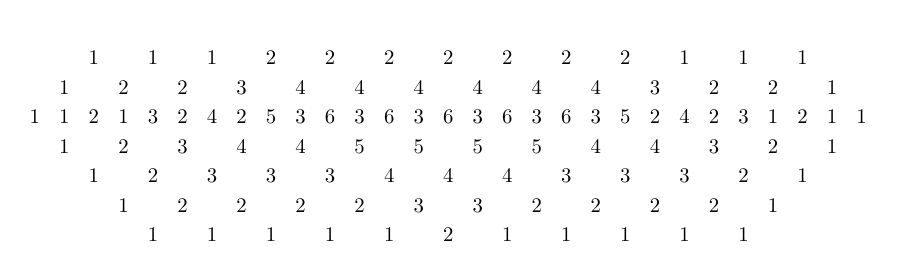}  
}
\caption[The starting functions $s_{P(e)}$]{The starting functions $s_{P(e)}$.}
\label{fig:startingfunction}
\end{figure}
\end{center}

For an \textbf{indecomposable} quiver representation $M \in \rep(Q)$, we can define the starting functions for each vertex $i \in v(Q)$:
$$s_{P(i)}(M):= \dim_K \Hom(P_i,M) = \dim_K M_i.$$
The order of $M$ is defined as the maximum:
$$\ord(M):= \max_{i \in v(Q)} s_{P(i)}(M).$$
When the quiver $Q$ is of type $E$, let $e \in v(Q)$ denote the branch vertex, then $\ord(M)=\orde(M)$ unless $\orde(M)=0$. The starting functions are detailed in \cite{bongartz1984critical}. For reference, Figure \ref{fig:startingfunction} presents a subset pertinent to this work.

The next lemma shows the affine paving property for representations of small order.

\begin{lemma}[{Follows \cite[Lemma 2.23]{maksimau2019flag}}]\label{lem:smallvecdim}
 For an indecomposable representation $M \in \rep(Q)$ with $\ord(M) \leqslant 2$, the variety $\Grr_{\dimvec{f}}(\Phi(M))$ is either empty or a direct product of some copies of $\mathbb{P}^1$. Especially, the partial flag variety $ \Grr_{\dimvec{f}}(\Phi(M))$ has an affine paving.
\end{lemma}
\begin{proof}\belowdisplayskip=-12pt
	For every $i \in v(Q)$, $\dim_K M_i\leqslant 2$. Since $Q$ is a tree and $M$ is indecomposable, for every $b\in a(Q)$ satisfying $\dim_K M_{s(b)}= \dim_K M_{t(b)}=2$, the map $M_{s(b)} \longrightarrow M_{t(b)}$ is an isomorphism. Therefore, when $\Grr_{\dimvec{f}}(\Phi(M)) \neq \varnothing$,\footnote{This condition imposes very strong restrictions on $\dimvec{f}$.} we get the natural embedding
	$$\Grr_{\dimvec{f}}(\Phi(M)) \longrightarrow \prod_{\substack{i \in v(Q) \;s.t.\\ \dim_K M_i=2 \\ \dimvec{f}_{(i,r)}=1 \text{ for some }r}} \mathbb{P}^1,$$\\[0.3cm]
	and the information of non-vertical arrows in the extended quiver (see Example \ref{eg:ext_quiver}) just reduce the number of $\mathbb{P}^1$. Precisely, one needs to carefully discuss three cases of $M_i \longrightarrow M_j$:
	\[\begin{tikzcd}
		K & {K^2} & {K^2} & K &[-0.5cm] {\text{and}} &[-0.5cm] {K^2} & {K^2.}
		\arrow["\cong", from=1-6, to=1-7]
		\arrow[two heads, from=1-3, to=1-4]
		\arrow[hook, from=1-1, to=1-2]
	\end{tikzcd}\]
\end{proof}

\begin{proof}[{Proof of Theorem \ref{thm:Dynkincase}, assuming Theorem \ref{thm:bigorder}}]
First of all, for any indecomposable representation $M \in \rep(Q)$ we obtain an affine paving. This follows from Theorem \ref{thm:bigorder} when $\ord(M)>2$, and follows from Lemma \ref{lem:smallvecdim} when $\ord(M)\leqslant 2$.

The general case follows by induction on the dimension vector. The indecomposable representations $\{N_i\}_{i \in Q_0}$ of quiver $Q$ can be ordered such that $[N_i,N_j]=0$ for all $i>j$. Therefore, every non-indecomposable representation $M$ can be decomposed as the direct sum of two nonzero representations $M_1, M_2$ satisfying $[M_2,M_1]^1=0$.
By applying Theorem \ref{thm:main1} to the short exact sequence 
$$0 \longrightarrow M_1 \longrightarrow M \longrightarrow M_2 \longrightarrow 0,$$
we get an affine paving from the affine pavings of $M_1$ and $M_2$, see Remark \ref{rem:topo}.
\end{proof}
\begin{remark}\label{rem:smallorder}
By the same technique one can show that, for Dynkin quiver $Q$ and any representation $M$ with $\displaystyle\max_{i \in v(Q)} \dim_K M_i \leqslant 2$, the variety $\Grr_{\dimvec{f}}(\Phi(M))$ has an affine paving. This result does not depend on Theorem \ref{thm:bigorder}.
\end{remark}

\section{Affine paving for big order representations}\label{sec:proofcomplement}

This section aims to establish the following theorem:

\begin{theorem}\label{thm:bigorder}
Suppose $Q$ is of Dynkin type. For any indecomposable representation $M \in \rep(Q)$ with $\ord(M)>2$, the (strict) partial flag variety $\Grr(\Phi(M))$ has an affine paving.
\end{theorem}

When the quiver $Q$ is of type $A$ or $D$, Theorem \ref{thm:bigorder} is trivially true since no indecomposable representation can have order bigger than two. So we only concentrate on type $E$.

The idea of the proof is as follows. For any indecomposable representation $Y$ with $\ord(Y)>2$, we put $Y$ into a short exact sequence 
$$\eta:0\longrightarrow X \longrightarrow Y \longrightarrow S \longrightarrow 0$$
fulfilling the assumptions of Theorem \ref{thm:main2}, and then $\Gr(\Phi(Y))$ has an affine paving if $\Img \Psi$ has. If additionally the map $X \hookrightarrow Y$ is a minimal sectional mono, then $\Img \Psi_{\dimvec{f},\dimvec{g}}$ can be written as the product space, which makes $\Img \Psi$ easier to understand, see Figure \ref{figprod}.

\begin{figure}[ht]
    \includegraphics[scale=0.7]{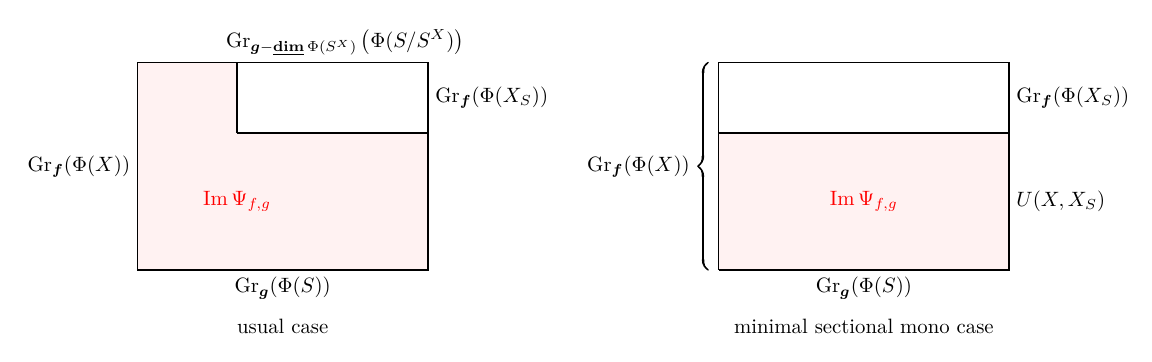}
    \caption{Image in the product space.}
    \label{figprod}
\end{figure}

\begin{defn}[Sectional morphism]\label{def:sec_morphism}
Let $Q$ be a quiver of Dynkin type, and $M,N \in \rep(Q)$ be two indecomposable representations of $Q$. A morphism $f \in \Hom_{KQ}(M,N)$ is called sectional if $f$ can be written as the composition
$$f: M=X_0 \stackrel{f_1}{\longrightarrow} X_1 \stackrel{f_2}{\longrightarrow} \cdots \stackrel{f_{t-1}}{\longrightarrow} X_{t-1} \stackrel{f_t}{\longrightarrow} X_t=N$$
where $f_i \in \Hom_{KQ}(X_{i-1},X_i)$ are irreducible morphisms between indecomposable representations, and $\tau X_{i+2} \ncong X_i$ for any suitable $i$.

A sectional morphism $f \in \Hom_{KQ}(M,N)$ is called as a sectional mono if $f$ is injective; a sectional mono is called minimal if $f_t \circ \cdots \circ f_{i+1}:X_i \longrightarrow N$ are surjective for any $i \in \{1,2,\ldots,t \}$.
\end{defn}

\begin{lemma}[Happel--Ringel]\label{lem:Happel--Ringel}
Let $M$ and $N$ be two indecomposable $Q$-representations. Any sectional morphism $f \in \Hom_{KQ}(M,N)$ is either surjective or injective.
\end{lemma}
\begin{proof}
When $Q$ is a quiver without oriented cycles, then $[N,M]^1 \leqslant [M,\tau N]=0$, thus by \cite[Lemma 7]{irelli2019cell} we get the result; when $Q$ is of type $\tilde{A}$, the result comes from \cite[Lemma 51]{irelli2019cell}.
\end{proof}

The next two lemmas tell us the existence of the desired short exact sequence.
\begin{lemma}\label{lem:msm&order}
	For every indecomposable representation $Y$ of type $E$ with $\ord(Y)>2$, there is a minimal sectional mono $f:X \longrightarrow Y$ such that $\dimv(X)<\dimv(Y)$.
\end{lemma}
\begin{proof}
	Suppose that $Y$ is an indecomposable representation of type $E$ such that $\ord(Y)>2$, then $s_{P(e)}(Y)=\ord(Y)>2$. By direct inspection on every table on Figure \ref{fig:startingfunction}, one sees that there is a sectional
	path (which can hence be chosen to be minimal) ending
	in $[Y]$ and starting at some $[M]$ with $s_{P(e)}(M)<s_{P(e)}(Y)$, so $\dimv(M)<\dimv(Y)$. Write this sectional path as $M=X_0 \rightarrow \cdots \rightarrow X_t=Y$, and take the maximal $i$ such that $\dimv(X_i)<\dimv(Y)$. Take $X=X_i$.
	By Lemma \ref{lem:Happel--Ringel}, the sectional mono $f:X \longrightarrow Y$ is injective. It follows that there exists a minimal sectional mono from $X$ to $Y$. 
\end{proof}

\begin{center}
\begin{figure}[ht]
\subcaptionbox[$E_6$]{$E_6$}[50mm][c]
{
\centering
\vspace{0mm}
\includegraphics[scale=0.8]{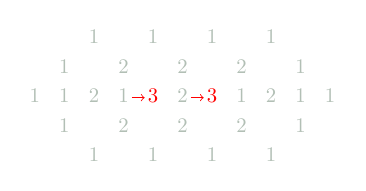}  
\vspace{1mm}
}~
\subcaptionbox[$E_7$]{$E_7$}[80mm][c]
{
\centering
\vspace{0mm}
\includegraphics[scale=0.8]{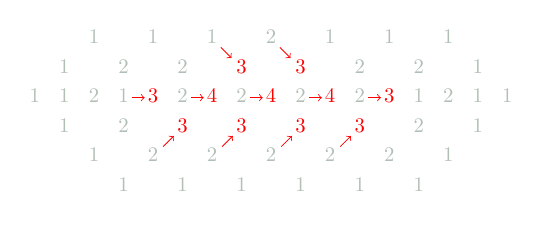}  
}
\subcaptionbox[$E_8$ usual cases]{$E_8$ usual cases}[130mm][c]
{
\centering
\vspace{0mm}
\includegraphics[scale=0.8]{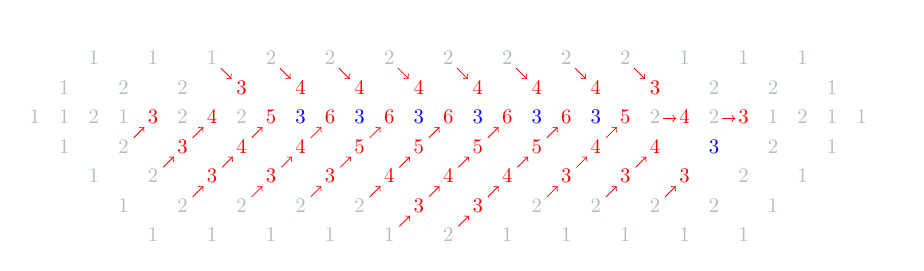}  
}
\subcaptionbox[$E_8$ exception cases]{$E_8$ exception cases\label{subfig:exception}}[130mm][c]
{
\centering
\vspace{0mm}
\includegraphics[scale=0.8]{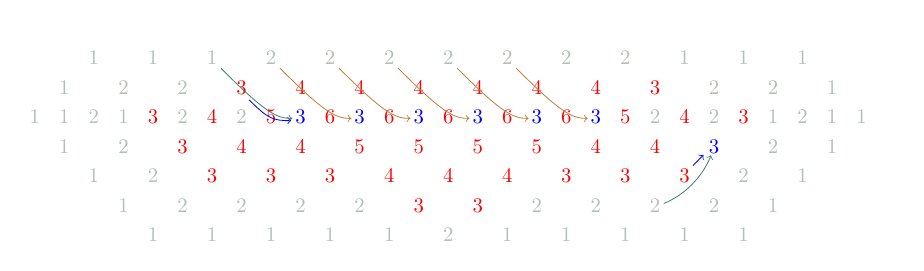}  
}
			\caption{Minimal sectional monos.}
			\label{fig:minisecmono}
\end{figure}
\end{center}

\begin{remark}
	The chosen minimal sectional monos are displayed in Figure \ref{fig:minisecmono}, where the arrows start at $[X]$ and end at $[Y]$. From
	those tables it is clear that in type $E_6$ the only indecomposable
	representations with order strictly bigger than two are
	those lying in the $\tau^-$-orbit of $P(e)$.
	In particular, they are the middle term of an almost split
	sequence, and hence in type $E_6$ Theorem \ref{thm:Dynkincase} is clear. Additionally, the tables demonstrate that the condition $\ord(Y) > 2$ in the lemma is essential and cannot be omitted.
\end{remark}
\begin{lemma}[{Adapted from \cite[Lemma 56 \& Lemma 8]{irelli2019cell}}]\label{lem:value}
Let $\iota: X \hookrightarrow Y$ be a minimal sectional mono, and $S:=Y/X$ be the quotient. Then $S$ is indecomposable and rigid, the short exact sequence 
$$\eta: 0 \longrightarrow X \longrightarrow Y \longrightarrow S \longrightarrow 0$$
generates $\Ext^1(S,X)$, and $S^X=S$. Moreover, one gets
\begin{equation*}
\left\{
\begin{aligned}
  \;& [X,X]=[Y,Y]=[S,S]=1 \\ 
  \;& [X,Y]=[Y,S]=[S,X]^1=1 \\ 
  \;& [X,S]=[Y,X]=[S,Y]=[S,X]=0 \\ 
  \;& [X,X]^1=[Y,Y]^1=[S,S]^1=0 \\ 
  \;& [X,Y]^1=[Y,S]^1=[X,S]^1=[Y,X]^1=[S,Y]^1=0 \\   
\end{aligned}
\right.
\end{equation*}
\end{lemma}
\begin{proof}
Since every indecomposable representation of Dynkin quiver is a brick and rigid, we get $[X,X]=[Y,Y]=1$ and $[X,X]^1=[Y,Y]^1=0$. Since $Q$ is of type $E$, by the definition of minimal sectional mono, we get $[Y,X]^1=0$. Then by Unger’s lemma \cite{irelli2019cell} one
gets that $S$ is indecomposable and $[S,X]^1 = 1$. In particular
$\eta$ is generating. Since $\iota$ is minimal, the pullback $j^*(\eta)$ by any
(proper) monomorphism $j:N \hookrightarrow S$ splits; it follows that
$S^X = S$. By applying the functors $[X,-],[Y,-],[-,X],[-,Y],[-,S]$ to the short exact sequence $\eta$ we get all the dimensions. 
\end{proof}

\begin{lemma}[{Follows \cite[Lemma 36]{irelli2019cell}, with the same proof}]\label{lem:descriptionofX_S}
	The notation $X$, $Y$ and $S$ remains as in the previous lemma. Let $X_1$ be the indecomposable module in the sectional path $X=X_0 \rightarrow \cdots \rightarrow X_t=Y$, and let $E \longrightarrow X$ be the minimal right almost split morphism ending in $X$. We have the decomposition $E=E' \oplus \tau X_1$ for some $Q$-representation $E'$. When $Y$ is not projective, $X_S$ is isomorphic to $\ker (E \longrightarrow \tau Y) \cong E' \oplus \ker (\tau X_1 \longrightarrow \tau Y)$; when $Y$ is projective, $X_S \cong E$.
\end{lemma}

\begin{proof}
Let $f: X \longrightarrow \tau S$ be a non-zero morphism, then by Lemma \ref{lemma:second_description_XS}, $X_S=\ker(f)$. 

When $X$ is not projective, we get a commutative diagram with exact rows:
\[\begin{tikzcd}
	{\eta_X:} & 0 & {\tau X} & E & X & 0 \\
	& 0 & {\tau X} & {\tau Y} & {\tau S} & 0
	\arrow[from=1-2, to=1-3]
	\arrow[from=1-3, to=1-4]
	\arrow[Rightarrow, no head, from=1-3, to=2-3]
	\arrow[from=1-4, to=1-5]
	\arrow[from=1-4, to=2-4]
	\arrow[from=1-5, to=1-6]
	\arrow["f", from=1-5, to=2-5]
	\arrow[from=2-2, to=2-3]
	\arrow[from=2-3, to=2-4]
	\arrow[from=2-4, to=2-5]
	\arrow[from=2-5, to=2-6]
\end{tikzcd}\]
where the commutativity comes from the fact that $[E,\tau S]=1$ (by applying $[-,\tau S]$ to $\eta_X$). By the snake lemma, one gets 
\begin{equation}\label{eq:X_S}
X_S = \ker(f) \;\cong\; \ker (E \longrightarrow \tau Y) \;\cong\; E' \oplus \ker (\tau X_1 \longrightarrow \tau Y). \tag{1}
\end{equation}

When $X=P(k)$ is projective while $Y$ is not projective, we get another commutative diagram with exact rows
\[\begin{tikzcd}
	{\phantom{\eta_X:}} & 0 & E & {X=P(k)} & {S(k)} & 0 \\
	& 0 & {\tau Y} & {\tau S} & {I(k)} & 0
	\arrow[from=1-2, to=1-3]
	\arrow[from=1-3, to=1-4]
	\arrow[from=1-3, to=2-3]
	\arrow[from=1-4, to=1-5]
	\arrow["f", from=1-4, to=2-4]
	\arrow[from=1-5, to=1-6]
	\arrow[hook', from=1-5, to=2-5]
	\arrow[from=2-2, to=2-3]
	\arrow[from=2-3, to=2-4]
	\arrow[from=2-4, to=2-5]
	\arrow[from=2-5, to=2-6]
\end{tikzcd}\]
one still gets \eqref{eq:X_S} through the snake lemma.

When $X=P(k)$ is projective and $Y=P(j)$ is projective, it is clear that $X \longrightarrow Y$ is irreducible, and one can check by hand that $[S,X/E]^1 = [S, S(k)]^1=1$, so $X_S \cong E$.

\end{proof}

\begin{corollary}
	When $\iota: X \longrightarrow Y$ is irreducible monomorphism, the representation $X_S$ is either $0$ or an indecomposable representation with property that $X_S \longrightarrow X$ is also an irreducible monomorphism.
\end{corollary}
\begin{proof}
If $\iota$ is an irreducible monomorphism, it follows that $X_1 = Y$, and $X$ does not belong to the $\tau^-$-orbit of $P(e)$. It follows that the representation $E$ has at most two indecomposable components. By Lemma \ref{lem:descriptionofX_S}, we obtain $X_S \cong E'$: 
\begin{itemize}
\item If $Y$ is not projective, then $\tau X_1 = \tau Y$, so $X_S \cong E' \oplus \ker(\tau X_1 \to \tau Y) \cong E'$;
\item If $Y$ is projective, then $\tau X_1 = 0$, so $X_S \cong E \cong E'$.
\end{itemize}
In both cases, $E' \subset E$ is either $0$ or indecomposable. Since $E \to X$ is the minimal right almost split morphism, $E' \to X$ is an irreducible monomorphism.
\end{proof}

For convenience, we simplify the notations: write $\Gr_{\dimvec{f}}(\Phi(M))$ as $\Gr(M)$, $\Gr_{\dimvec{f}}(\Phi(M)) \setminus \Gr_{\dimvec{f}}(\Phi(N))$ as $U(M,N)$, where we omit subscripts which indicate the dimension vectors.

\begin{lemma}[{Follows \cite[Theorem 59]{irelli2019cell}}]\label{lem:worst_case}
Let $f:X \hookrightarrow Y$ be a minimal sectional mono and $S:=Y/X$ be the quotient. When $X_S=F \oplus T$ with $F$ and $T$ nonzero indecomposable, $F\hookrightarrow X$ irreducible and $T \hookrightarrow X/F$ sectional mono, we have
$$U(X,X_S)\longrightarrow \Gr(F) \shorttimes U(X/F,T) \oder U(F,F_{X/F})$$
as a (Zariski) locally-trivial affine bundle.
\end{lemma}

\begin{proof}
We have two short exact sequences satisfying the conditions in \ref{thm:main2}:
\begin{center}
\begin{tikzcd}[row sep=0mm]
\eta: & 0 \arrow[r] & F \arrow[r] & X \arrow[r, "\pi"]    & X/F \arrow[r]   & 0 \\
\xi:  & 0 \arrow[r] & T \arrow[r] & X/F \arrow[r, "\pi'"] & X/X_S \arrow[r] & 0.
\end{tikzcd}
\end{center}
Let $N \in \Gr(X)$ be a subrepresentation, it is obvious that $N \in \Gr(X_S) \iff \pi'\circ \pi(N)=0$, so 
  \begin{equation*}
  \begin{aligned}
  N\in U(X,X_S) & \iff \pi'\circ \pi(N) \neq 0\\
  & \iff \pi(N) \notin \Gr(T)\\
  & \iff \pi(N) \in U(X/F,T)\\
  & \iff \Psi_{\eta}(N) \in \Gr(F) \times U(X/F,T).
  \end{aligned}
  \end{equation*}
Thus the (Zariski) locally-trivial affine bundle map
$$U(X,F) \longrightarrow \Gr(F) \shorttimes \Gr(X/F)$$
restricted to the (Zariski) locally-trivial affine bundle map

\vspace{\abovedisplayskip}
\hfill $\displaystyle U(X,X_S) \longrightarrow \Gr(F) \times U(X/F,T).$ \qedhere
\end{proof}

We conclude Theorem \ref{thm:main2}, Lemma \ref{lemma:second_description_XS}, \ref{lem:value}, \ref{lem:descriptionofX_S}, \ref{lem:worst_case} in the following proposition.

\begin{proposition}\label{prop:conclusion_for_combinatorics}
Let $f:X \hookrightarrow Y$ be a minimal sectional mono and $S:=Y/X$ be the quotient. One gets (Zariski) locally-trivial affine maps
\begin{equation*}
\begin{aligned}
\Gr(Y) &\longrightarrow \Gr(X) \shorttimes \Gr(S) \oder U(X,X_S), \\
U(Y,X) &\longrightarrow \Gr(X) \shorttimes \Gr(S) \oder U(X,X_S), \\
\end{aligned}
\end{equation*}
where
\begin{equation*}
\begin{aligned}
  X_S:=\;& \max \left\{ M \subseteq X \,\middle|\; [S,X/M ]^1=1 \right\} \\ 
  =\;& \ker (X \longrightarrow \tau S) \\ 
  =\;& \begin{cases}
  E' \oplus \ker (\tau X_1 \longrightarrow \tau Y) & Y \text{ not projective, }\\
  E & Y \text{ projective. }
  \end{cases} \\ 
\end{aligned}
\end{equation*}
Moreover, when $X_S=F \oplus T$ with $F$ and $T$ nonzero indecomposable, $F\hookrightarrow X$ irreducible and $T \hookrightarrow X/F$ sectional mono, one gets a (Zariski) locally-trivial affine map
$$U(X,X_S)\longrightarrow \Gr(F) \shorttimes U(X/F,T) \oder U(F,F_{X/F}).$$
\end{proposition}

With Proposition \ref{prop:conclusion_for_combinatorics} established, we prove Theorem \ref{thm:bigorder} case by case, reducing it to a purely combinatorial problem. For each indecomposable representation $Y$ with $\ord(Y) > 2$, one can trace a path in Figure \ref{fig:minisecmono} ending at $Y$ that corresponds to the chosen minimal sectional monomorphism $\iota: X \longrightarrow Y$. If there are two paths terminating at $Y$, we select the shorter path when it is a sectional mono; otherwise, the longer path is chosen as the sectional mono.

\begin{proposition}\label{prop:irrcase}
Let $\iota: X \hookrightarrow Y$ be an irreducible monomorphism in Figure \ref{fig:minisecmono}. Then $\Gr(Y)$ has an affine paving.
\end{proposition}

\begin{proof}
Let $S=Y/X$. By applying Theorem \ref{thm:main2} to the short exact sequence 
$$0\longrightarrow X \longrightarrow Y \longrightarrow S \longrightarrow 0,$$
we get a (Zariski) locally-trivial affine map
$$\Gr(Y) \longrightarrow \Gr(X) \shorttimes \Gr(S) \oder U(X,X_S).$$ 
By observation of  Figure \ref{fig:minisecmono}, $\orde(S)=\orde(Y)-\orde(X)$ is smaller or equal to $2$, so by Lemma \ref{lem:smallvecdim} $\Gr(S)$ has the affine paving property. Let $Y':=X$, $X':=X_S$, $S':=Y'/X'$, we apply Theorem \ref{thm:main2} to the short exact sequence 
$$0\longrightarrow X' \longrightarrow Y' \longrightarrow S' \longrightarrow 0$$
and get (Zariski) locally-trivial affine maps
\begin{equation*}
\begin{aligned}
\Gr(X) &\longrightarrow \Gr(X') \shorttimes \Gr(S') \oder U(X',X'_{S'}) \\
U(X,X_S)&\longrightarrow \Gr(X') \shorttimes \Gr(S') \oder U(X',X'_{S'}).
\end{aligned}
\end{equation*}
Luckily $\orde(S')$ is still smaller or equal to $2$. We can continue this process until the order of representations is small enough.
\end{proof}

\begin{remark}
	We can not copy everything in \cite[Lemma 56]{irelli2019cell}, sometimes it would happen that $X_S=F \oplus T$ with $F$ and $T$ indecomposable, $F \hookrightarrow X$ is irreducible but $T \longrightarrow X/F$ is not a sectional mono. Even when $X_S$ is indecomposable, the map $X_S \longrightarrow X$ can be not a sectional mono.
	
	For example, take the quiver of type $E_7$: 
	\[
	\begin{tikzcd}[row sep=3mm, column sep=5mm]
	                &                 &                 & \bullet \arrow[d] &                 &                 \\
	\bullet \arrow[r] & \bullet \arrow[r] & \bullet \arrow[r] & \bullet           & \bullet \arrow[l] & \bullet \arrow[l]
	\end{tikzcd}
	\]
	 take $Y=\representation{111111}{0}$, $X=\representation{111110}{0}$, then $X_S=\representation{111100}{0}$, the map $X_S \longrightarrow X$ is not a sectional mono.
	 
\end{remark}

Proposition \ref{prop:irrcase} solved all the cases for type $E_6$, $E_7$, and the rests are exception cases in type $E_8$, as shown in Figure \ref{fig:minisecmono} \subref{subfig:exception}). These exceptional cases are addressed in a similar manner, although the argument becomes more intricate.

\begin{center}
\begin{figure}[ht]
\centering
\vspace{0mm}
\includegraphics[scale=0.8]{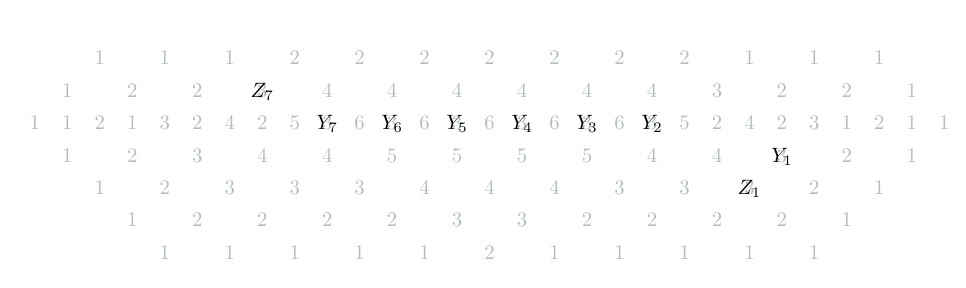}  
			\caption{Labeling for $E_8$ exception cases.}
			\label{fig:hZi}
\end{figure}
\end{center}

As illustrated in Figure \ref{fig:hZi}, the seven ending vertices $[Y]$ are labeled as $Y_1, \ldots, Y_7$, arranged from right to left. In the following examples, we demonstrate that each $\Gr(Y_i)$ admits an affine paving. Together with Proposition \ref{prop:irrcase}, this establishes Theorem \ref{thm:bigorder}.
\begin{center}
\begin{figure}[ht]
\centering
\vspace{0mm}
\includegraphics[scale=0.8]{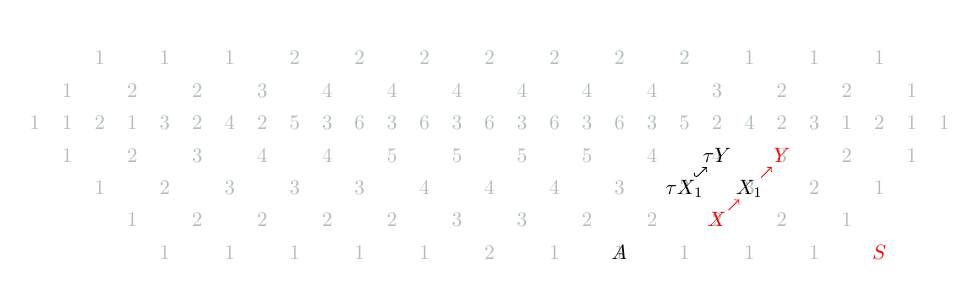}  
			\caption{Labeling in the case $Y=Y_1$.}
			\label{fig:h1}
\end{figure}
\end{center}
\begin{eg}
We demonstrate that $\Gr(Y_1)$ admits an affine paving. When $Z_1 \longrightarrow Y_1$ is injective, we know from Proposition \ref{prop:irrcase} that $\Gr(Y_1)$ has an affine paving. When $Z_1 \longrightarrow Y_1$ is not injective, we are in the situation of Figure \ref{fig:h1}, where $Y=Y_1$ and $A=X_S$. By applying Theorem \ref{thm:main2} to the short exact sequences
\begin{equation*}
\begin{aligned}
&0\longrightarrow X \longrightarrow Y \longrightarrow Y/X \longrightarrow 0\\
&0\longrightarrow A \longrightarrow X \longrightarrow X/A \longrightarrow 0\\
\end{aligned}
\end{equation*}
we obtain (Zariski) locally-trivial affine maps 
\begin{equation*}
\begin{aligned}
\Gr(Y) &\longrightarrow \Gr(X) \shorttimes \Gr(X/Y)\hspace{3mm} &&\oder U(X,A)\\
U(X,A) &\longrightarrow \Gr(A) \shorttimes \Gr(X/A) &&\oder U(A,-).\footnotemark\\
\end{aligned}
\end{equation*}
\footnotetext{$\Gr(A)$ is empty or a singleton, so is $U(A,-)$, no matter what representation is in the dash mark.}
These maps give the variety $\Gr(Y)$ an affine paving from bottom to top.
\end{eg}

In the examples that follow, the chosen short exact sequences are omitted for clarity. As a reminder to the reader, for a short exact sequence 
$$0\longrightarrow M \longrightarrow N \longrightarrow N/M \longrightarrow 0$$
where $\iota: M \hookrightarrow N$ is a minimal sectional mono, Theorem \ref{thm:main2} implies the existence of (Zariski) locally-trivial affine maps 
\begin{equation*}
\begin{aligned}
\Gr(N) &\longrightarrow \Gr(M) \shorttimes \Gr(N/M)\hspace{3mm} &&\oder U(M,M_{N/M})\\
U(N,M) &\longrightarrow \Gr(M) \shorttimes \Gr(N/M) &&\oder U(M,M_{N/M})\\
U(N,M \oplus T) &\longrightarrow \Gr(M) \shorttimes U(N/M,T) &&\oder U(M,M_{N/M}).\\
\end{aligned}
\end{equation*}

\begin{center}
\begin{figure}[ht]
\centering
\vspace{0mm}
\includegraphics[scale=0.8]{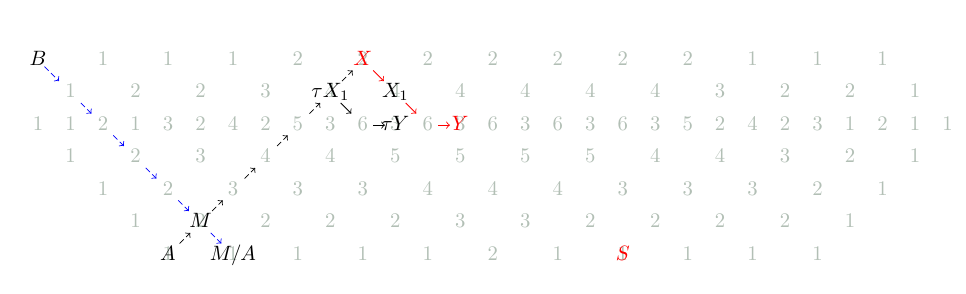}  
			\caption{Labeling in the case $Y=Y_5$.}
			\label{fig:h2}
\end{figure}
\end{center}
\begin{eg}
We demonstrate that $\Gr(Y_5)$ admits an affine paving; the same arguments extend to $\Gr(Y_i)$ for $i=2,3,4$. Notations from Figure \ref{fig:h2} are used, where $Y=Y_5$ and $A=X_S$. We have
$$\Gr(Y) \longrightarrow \Gr(X) \shorttimes \Gr(Y/X) \oder U(X,A).$$
When the map $M \longrightarrow X$ is not monomorphism, we get
$$U(X,A) \longrightarrow \Gr(A) \shorttimes \Gr(X/A) \oder U(A,-);$$
when the map $M \longrightarrow X$ is monomorphism, we get
\begin{equation*}
\begin{aligned}
&U(X,A)=U(X,M) \bigsqcup U(M,A)\hspace{-30mm}&&\\
U(X,M) &\longrightarrow \Gr(M) \shorttimes \Gr(X/M) && \oder U(M,A \oplus B)\\
U(M,A) &\longrightarrow \Gr(A) \shorttimes \Gr(M/A)&& \oder U(A,-)\\
U(M,A \oplus B)  &\longrightarrow \begin{cases}
\Gr(A) \shorttimes \Gr(M/A)\\
\Gr(A) \shorttimes U(M/A,B)\\
\end{cases} \hspace{-5mm}&& \begin{aligned}
\oder U(A,-), &\qquad B=0,\\
\oder U(A,-), &\qquad B\neq 0.
\end{aligned}\\
\end{aligned}
\end{equation*}

Since the order of $X$, $Y/X$, $A$, $X/A$, $M$, $X/M$, $M/A$ is smaller or equal to $2$ as well as $\ord(A)=\ord(M/A)=1$, the induction process stops, we get an affine paving of $\Gr(Y)$.
\end{eg}

\begin{center}
\begin{figure}[ht]
\centering
\vspace{0mm}
\includegraphics[scale=0.8]{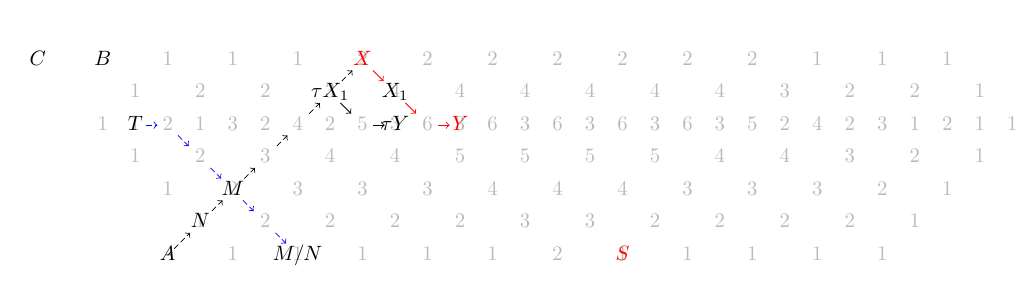}  
			\caption{Labeling in the case $Y=Y_6$.}
			\label{fig:h3}
\end{figure}
\end{center}

\begin{eg}
We verify that $\Gr(Y_6)$ is equipped with an affine paving. Referring to Figure \ref{fig:h3}, where $Y=Y_6$ and $A=X_S$, we obtain the map 
$$\Gr(Y) \longrightarrow \Gr(X) \shorttimes \Gr(Y/X) \oder U(X,A).$$
If $A=0$, we're done; if not, we decompose $A \longrightarrow Y$ as compositions of minimal sectional monos:\\
Case 1: $M\longrightarrow X$ is not injective, then
 \begin{equation*}
 \begin{aligned}
 &U(X,A)=U(X,N) \bigsqcup U(N,A)\\
 U(X,N) &\longrightarrow \Gr(N) \shorttimes \Gr(X/N) \oder U(N,-)\\
 \end{aligned}
 \end{equation*}
Case 2: $M\longrightarrow X$ is injective, then
  \begin{equation*}
  \begin{aligned}  
  U(X,A)&=U(X,M) \bigsqcup U(M,N) \bigsqcup U(N,A)\hspace{-20mm}&&\\
  U(X,M) &\longrightarrow \Gr(M) \shorttimes \Gr(X/M) &&\hspace{-3mm} \oder U(M,N \oplus T)\\
  U(M,N) &\longrightarrow \Gr(N) \shorttimes \Gr(M/N)&&\hspace{-3mm} \oder U(N,-).\\
  \end{aligned}
  \end{equation*}
  Since $M \longrightarrow X$ is injective, we get $B=0$, thus $C=0$ also, and then the sectional map $T \longrightarrow M/N$ is injective. We get 
  $$  U(M,N \oplus T)  \longrightarrow 
    \Gr(N) \shorttimes U(M/N,T)
    \oder U(N,-).$$
  Since $\Gr(X)$, $\Gr(Y/X)$, $\Gr(N)$, $\ldots$ have the property of affine paving, we conclude that $\Gr(Y)$ also has the property of affine paving.
\end{eg}

\begin{center}
\begin{figure}[ht]
\centering
\vspace{0mm}
\includegraphics[scale=0.8]{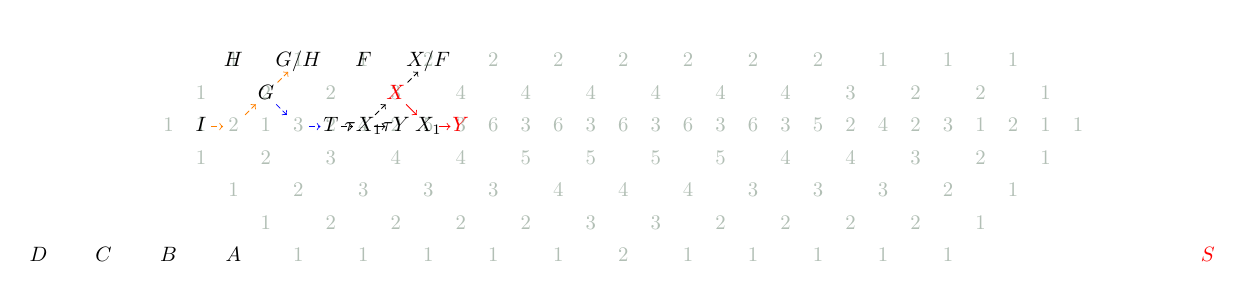}  
			\caption{Labeling in the case $Y=Y_7$.}
			\label{fig:h4}
\end{figure}
\end{center}

\begin{eg}
Finally we begin to tackle the most difficult case(Figure \ref{fig:h4}), where we show that $\Gr(Y_7)$ has an affine paving. 

When $Z_7 \longrightarrow Y_7$ is not injective, the map $F \longrightarrow Y_7$ is the minimal sectional mono, using the short exact sequences
$$0\longrightarrow F \longrightarrow Y \longrightarrow Y/F \longrightarrow 0,$$ 
 we get
$$\Gr(Y) \longrightarrow \Gr(F) \shorttimes \Gr(Y/F) \oder U(F,-),\hspace{8mm}$$
and then we get the result.

When $Z_7 \longrightarrow Y_7$ is injective (i.e., $A=0$), we compute that $X_S = F \oplus T$, yielding the map
$$\Gr(Y) \longrightarrow \Gr(X) \shorttimes \Gr(Y/X) \oder U(X,F \oplus T).$$
Notice that

$$A=0 \hspace{4mm}\Longrightarrow \begin{cases}
\hspace{4mm}B=0  \hspace{4mm}\Longleftrightarrow\hspace{4mm} T  \longrightarrow X/F \text{ is injective}\\
\hspace{4mm}C=0  \hspace{4mm}\Longleftrightarrow\hspace{4mm} G  \longrightarrow T \text{ is injective}\\
\hspace{4mm}D=0  \hspace{4mm}\Longleftrightarrow \hspace{4mm}I  \longrightarrow G/H \text{ is injective}\\
\end{cases}$$
We get 
\begin{equation*}
\begin{aligned}
U(X,F \oplus T) &\longrightarrow \Gr(F) \shorttimes U(X/F,T) &&\oder U(F,-)\\
U(X/F,T) &\longrightarrow \Gr(T) \shorttimes \Gr(X/X_S) &&\oder U(T,G)\\
U(T,G) &\longrightarrow \Gr(G) \shorttimes \Gr(T/G) &&\oder U(G,H \oplus I)\\
U(G,H \oplus I) &\longrightarrow \Gr(H) \shorttimes U(G/H,I) &&\oder U(H,-). \\
\end{aligned}
\end{equation*}
We conclude that $\Gr(Y)$ has the affine paving property.\qedhere
\end{eg}

\section{Application: Affine Case}\label{sec:affine}
This section tries to explain the difficulty of the Conjecture \ref{conj:affinecase}.

\begin{conj}\label{conj:affinecase}
For any affine quiver $Q$ and any indecomposable representation $M \in \rep(Q)$, the (strict) partial flag variety $\Flag{}(M)\cong\Grr(\Phi(M))$ has an affine paving.
\end{conj}

Actually, if readers follow the proof in \cite[Section 6]{irelli2019cell}, and change everything from $\Gr(-)$ to $\Grr(\Phi(-))$, then there is no difference except the Proposition 48, in which the authors proved the affine paving properties of quasi-simple regular representations. So we reduced the question to the case of quasi-simple regular representation. Combined with Lemma \ref{lem:affineADcase}, we've proved the affine paving properties for $\tilde{A},\tilde{D}$ cases.

\begin{lemma}\label{lem:affineADcase}
	Assume that $Q$ is a quiver of type $\tilde{A}$ or $\tilde{D}$, $M \in \rep(Q)$ is the \textbf{regular quasi-simple} representation, then the Grassmannian $\Grr(\Phi(M))$ has an affine paving.
\end{lemma}
\begin{proof}
The concept ``quasi-simple'' is defined in \cite[Definition 15]{irelli2019cell}; the concepts ``preprojective'',``preinjective'' and ``regular'' are defined in \cite[2.1.1]{irelli2019cell}. It's shown in \cite[Section 9, Lemma 3]{crawley1992lectures} that the regular quasi-simple representation $M$ have dimension vector smaller or equal to the minimal positive imaginary root, thus $\orde(M) \leqslant 2$ for the quiver of type $\tilde{D}$ and $\orde(M) \leqslant 1$ for the quiver of type $\tilde{A}$.
\end{proof}

\begin{theorem}\label{thm:affinecase}\
\begin{enumerate}[(1)]
\item Assume that $Q$ is a quiver of type $\tilde{A}$ or $\tilde{D}$, then for any indecomposable representation $M$, the Grassmannian $\Gr(\Phi(M))$ has an affine paving;
\item Assume that $Q$ is an affine quiver of type $\tilde{E}$, and $\Gr(\Phi(N))$ has an affine paving for any regular quasi-simple representation $N \in \rep(Q)$. The Grassmannian $\Gr(\Phi(M))$ then has an affine paving for any indecomposable representation $M$.
\end{enumerate}
\end{theorem}

For a regular quasi-simple representation $Y$ of type $\tilde{E}$, it's possible that there's no short exact sequence
$$\eta:0\longrightarrow X \longrightarrow Y \longrightarrow S \longrightarrow 0$$
such that $[S,X]^1 \leqslant 1$. Then we can no longer use Theorem \ref{thm:main1} or \ref{thm:main2}. Hence, the new methods are needed for this case.


\bibliographystyle{plain}
\bibliography{affinepaving}

\end{document}